\documentclass[12pt]{article}
\usepackage{amssymb,amsmath,amsthm, amsfonts}
\usepackage{graphicx}
\usepackage{epsfig}
\usepackage{mathrsfs}
\usepackage{amsthm}

\theoremstyle{plain}
\newtheorem{theorem}{Theorem}[section]
\newtheorem{lemma}{Lemma}[section]

\theoremstyle{definition}

\numberwithin{equation}{section}

\begin{document}

\title{Global classical solutions to
 a  two-dimensional chemotaxis-fluid system involving signal-dependent degenerate diffusion}

\author{Yansheng Ma\\
School of Mathematics and Statistics, Northeast Normal University\\
Changchun, 130024, People's Republic of China
\\
  Peter Y.~H.~Pang\\
Department of Mathematics, National University of Singapore\\
10 Lower Kent Ridge Road, Republic of Singapore 119076
\\
Yifu Wang\thanks{
Corresponding author. Email: {\tt wangyifu@bit.edu.cn}}\\
School of Mathematics and Statistics, Beijing
Institute of Technology\\
Beijing 100081, People's Republic of China}
\date{}
\maketitle
\vspace{0.0cm}
\noindent
\begin{abstract}
This paper is concerned with the two-dimensional chemotaxis-fluid model
\begin{equation*}
\begin{cases}
n_t+u\cdot\nabla n=\Delta (n\phi(v))+\mu n(1-n),\\
v_t+u\cdot\nabla v=\Delta v-nv,\\
u_t+ \kappa (u\cdot\nabla) u=\Delta u+n\nabla\Phi-\nabla P, \quad\nabla\cdot u=0,
\end{cases}
\end{equation*}
  accounting for signal-dependent motilities of microbial populations
 interacting with  an incompressible liquid  through transport and buoyancy, where the suitably smooth function $\phi$ satisfies $\phi>0$ on $(0,\infty)$ with $\phi(0)=0$ and $\phi'(0)>0$, and the parameter $\mu\geq 0$.
For all reasonably regular initial data, if $\mu=0$,
the  corresponding initial boundary value problem possesses  global classical solutions with a smallness condition on $\int_\Omega n_0$; whereas  if $\mu>0$, this problem possesses global bounded classical solutions, which can converge toward (1,0,0) as time tends to infinity when a
certain small mass is imposed on the initial data $v_0$. These results extend recent results for the fluid-free system to one in a Navier-Stokes fluid environment.

\end{abstract}

\vspace{0.3cm}
\noindent {\bf\em Key words:}~Chemotaxis consumption model,  signal-dependent motility, Navier-Stokes, global existence.

\noindent {\bf\em 2020 Mathematics Subject Classification}:~35K55, 35K65, 35B40, 35Q35, 92C17.

\newpage
\section{Introduction}

This study is concerned with the chemotaxis-fluid model
\begin{equation}\label{1.1}
\begin{cases}
n_t+u\cdot\nabla n=\Delta (n\phi(v))+\mu n(1-n),\\
v_t+u\cdot\nabla v=\Delta v-nv,\\
u_t+ \kappa (u\cdot\nabla) u=\Delta u+n\nabla\Phi-\nabla P, \quad\nabla\cdot u=0,
\end{cases}
\end{equation}
for the  behavior of a microbial population
 interacting with an incompressible fluid through transport and buoyancy (\cite{Tuval(2005)}). In this model, the motility of the microbial population with density $n$ is assumed to depend on the concentration $v$ of a certain signal substance consumed by $n$ upon contact. As such, this study makes further progress on
recent interest and advances in the modeling literature addressing bacterial movement
in situations where cell motility is biased by chemical cues (\cite{FTLH,Liu}).
In addition to the signal-induced motility, the microbial population $n$ is assumed to interact with its fluid environment through the fluid's velocity field
$u$ and associated pressure $P$ via buoyant forces mediated by the given gravitational potential $\Phi =\Phi(x)$, while both the microbial population and the signal substance are subject to convection.

In contrast with the prevalent literature on chemotaxis-fluid systems (\cite{WinklerTAMS,WinklerCPDE,Cao}), there is an intriguing interrelation between the random diffusion and cross-diffusion in the system (1.1). Indeed, when splitting the nonlinear second order operator
 $\Delta (n\phi(v))$ into the diffusive part  $\nabla\cdot(\phi(v)\nabla n)$  and
 the cross-diffusive contribution $-\nabla\cdot(\psi(v) n\nabla v)$,
 one can see that the migration rate is precisely linked through the relation $\psi(v)=-\phi'(v)$, which indicates
 that the microbial population achieves its signal sensing in a distinctly local manner (\cite{Liu,FTLH}).

In order to understand the implication of this link between random diffusion and cross-diffusion,
we start with the extensive literature on the Neumann
problem for the Keller-Segel model
with signal-dependent motility 
\begin{equation}\label{1.2a}
\begin{cases}
n_t=\Delta (n\phi(v))+\mu n(1-n),\\
v_t=\Delta v-v+n,\\
\end{cases}
\end{equation}
in a smooth bounded domain $\Omega\subset \mathbb{R}^d$, $d\geq1$. For instance, when $ \mu=0$,  the motility function $\phi\in C^3((0,\infty))$  and has strictly positive upper and lower bounds,
the existence of a unique globally bounded classical solution  to \eqref{1.2a} was shown in \cite{TaoM3AS,XiaoJDE}. In particular,
a critical-mass phenomenon occurs for \eqref{1.2a} with $\phi(v) =e^{-\chi v}$ in the two-dimensional case (\cite{Fujie,JinPAMS}).
We remark that with certain initial data of super-critical mass, the
solution of \eqref{1.2a} becomes unbounded at time infinity, which differs from the finite-time blowup
behavior of the classical Keller-Segel model. When $\mu>0$ and $\displaystyle\lim_{v\rightarrow \infty} \phi(v)=0$, the dampening effect on the prevention of blow-up of \eqref{1.2a}
was  detected in \cite{JinWang,Xiao}.

When the signal substance is not produced but rather consumed by the microbial population, the signal-dependent model \eqref{1.2a} becomes
\begin{equation}\label{1.3a}
\begin{cases}
n_t=\Delta (n\phi(v))+\mu n(1-n),\\
v_t=\Delta v-nv,\\
\end{cases}
\end{equation}
which is the fluid-free counterpart of \eqref{1.1}.
 In contrast to \eqref{1.2a},  a distinguishing core hypothesis  underlying \eqref{1.3a} is that the signal $v$ is degraded  by the microbial cells upon contact and becomes substantially dimmed down. Accordingly,  the degeneracy near $v=0$ by $\phi(0)=0$
affects the solution behavior of \eqref{1.3a} at large time, possibly leading to the formation of patterns, such as structure-supporting features observed in some contexts of bacterial migration in nutrient-poor environments \cite{Leyva,Matsushita}. Indeed, in some non-degenerate cases determined by the strict positivity of 
the regular motility function $\phi$ on $[0,\infty)$,  solutions to the associated  Neumann
problem of \eqref{1.3a}  with $\mu=0$ stabilize to spatially homogeneous steady states (\cite{LiWinkler,PL}), whereas for degenerate motility functions
$\phi$ such as $\phi(v)=v^\alpha$ with suitable $ \alpha\geq 1$, the large time limits of the first solution component of \eqref{1.3a} with $\mu=0$ may deviate from being constant for some initial data \cite{WinklerAIHP,WinklerNonlinearity,LiLou}.

From the perspective of mathematical analysis, the degeneracy of signal-dependent type
brings about noticeable challenges
 already at the level of basic existence issue, as the smoothing effects of diffusion and the potentially destabilizing action of cross-diffusion are reduced simultaneously at small signal concentration in this circumstance. In this vein,
we note that the approaches coping with the diffusion-inhibiting mechanisms in \eqref{1.3a} are quite different from those for degeneracies of porous medium type. Indeed, since the degeneracy of the motility function $\phi$ in \eqref{1.3a}
 reduces a priori information on regularity to a significant extent,  the derivation of crucial associated estimates
requires us to take into account the intricate spatio-temporal nonlocal character of the cross-degeneracy in \eqref{1.3a}.
 Accordingly, even the issues on the global solvability of \eqref{1.3a} and its close variants seem far from evident, and some results on global solvability could only be achieved  in certain generalized frameworks by far, despite the fact that global bounded and even continuous solutions could be constructed in one- and two-dimensional
domains when the  degeneracy of $\phi$  near $0$ is sufficiently mild
 (\cite{WinklerM3AS,Winkleradv,WinklerZAMP,Ligenglin,Liangwang}).
For example,
when  $\phi(\xi)\sim \xi^\alpha $ for small values $\xi>0$, the existence of global very weak solutions
 of \eqref{1.3a}
  was proven for
all $\alpha>0$   in \cite{WinklerZAMP}. Furthermore, bounds
of  $ \int^T_0\int_{\Omega} n^2 \phi(v)$, the crucial regularity features of solutions to the associated
 problem of \eqref{1.3a}, were derived by exploiting the identity
\begin{equation}\label{1.4}
\partial_t B^{-1}n+ n\phi (v)= B^{-1}( n\phi (v)) \end{equation}
 where $B$ denotes a suitable realization of $-\Delta+1$.

 Recently some  progress has been achieved in the global solvability of  \eqref{1.2a} coupled with the Navier-Stokes equations
   \begin{equation}\label{1.5}
\begin{cases}
n_t+u\cdot\nabla n=\Delta (n\phi(v)),\\
v_t+u\cdot\nabla v=\Delta v+n-v,\\
u_t+ \kappa (u\cdot\nabla) u=\Delta u+n\nabla\Phi-\nabla P, \quad\nabla\cdot u=0,
\end{cases}
\end{equation}
in certain generalized sense, where the motility function $\phi$ is positive on $(0,\infty)$ but may decay algebraically at
large values (\cite{Tian,WinklerADE}).
 Due to the additional coupling to
the incompressible Navier-Stokes system, it seems that
 the analytical approach based on the corresponding
fluid perturbation of \eqref{1.4}
\begin{equation}\label{1.6}
\partial_t B^{-1}n+ n\phi (v)= B^{-1}( n\phi (v))-  B^{-1}(u\cdot\nabla n)
\end{equation}
is  no longer available.  Indeed, along with  the embedding inequality
$$\|B^{-1} \nabla \varphi
\|_
{
L^{\frac{2p}{2-p}}(\Omega)
}\leq C(p)\|\varphi\|_
{
L^{p}(\Omega)
}$$
in two dimensional  settings,
a priori knowledge on fluid regularity in the Stokes evolution equations (i.e. $\kappa=0$) allows the
contribution of $ B^{-1}(u\cdot\nabla n)$ to be suitably controlled by the corresponding  diffusive action
(see the proof of \cite [Lemma 4.3] {Tian}). However, as pointed in \cite{Winklerpreprint,WinklerADE}, in the case of $\kappa=1$,
available
a priori knowledge on fluid regularity seems insufficient to provide beneficial information about
 $ B^{-1}(u\cdot\nabla n)$; in particular, it is unclear to what extent
 $ \int_\Omega |B^{-\frac 12} n|^2 $
 retains crucial quasi-energy properties that
   provide bounds for $ \int^T_0\int_{\Omega} n^2 \phi(v)$.

 To the best of our knowledge, none of the results available so far seems applicable to
the chemotaxis-Navier-Stokes system with  degenerate motility function.
 The objective of this paper is to undertake a first step toward the qualitative analysis of \eqref{1.1}.

We shall consider \eqref{1.1} in a smooth bounded domain $\Omega\subset \mathbb{R}^2$ with the
 boundary conditions
  \begin{equation}\label{1.7}
 \frac{\partial n}{\partial \nu}=\frac{\partial v}{\partial \nu}=0\quad\hbox{and}\quad
  u=0\quad \hbox{for}\, x\in \partial\Omega \quad\hbox{and}\quad t>0,
\end{equation}
 and initial conditions
 \begin{equation}\label{1.8}
 n(x,0)=n_0(x),\,v(x,0)=v_0(x),\,u(x,0)=u_0(x),\quad x\in \Omega.
\end{equation}
We will assume throughout that the initial data satisfy the technically motivated regularity and positivity requirements
\begin{equation}\label{1.9}
\left\{\begin{array}{l}
n_0\in C^\alpha(\bar{\Omega})\mbox{ for some }\alpha\in(0,1),\mbox{ with }n_0\geq0\mbox{ and }n_0\not\equiv0\mbox{ in }\bar{\Omega},\\
v_0\in W^{1,\infty}(\Omega)\mbox{ with }v_0>0\mbox{ in }\bar{\Omega},\\
u_0\in D(A^\beta)\mbox{ for some }\beta\in(\frac{1}{2},1),
\end{array}
\right.
\end{equation}
where $A=-\mathcal{P}\Delta$ denotes the realization of the Stokes operator in $L^2(\Omega;\mathbb{R}^2)$ defined on the domain $D(A):=W^{2,2}(\Omega;\mathbb{R}^2)\cap W_0^{1,2}(\Omega;\mathbb{R}^2)\cap L_\sigma^{2}(\Omega;\mathbb{R}^2)$, with $L_\sigma^2(\Omega;\mathbb{R}^2):=\{\varphi\in L^{2}(\Omega;\mathbb{R}^2)|\nabla\cdot\varphi=0\}$ and $\mathcal{P}$ being the Helmholtz projection of $L^2(\Omega;\mathbb{R}^2)$ onto $L_\sigma^2(\Omega;\mathbb{R}^2)$.

As to the parameter functions in  \eqref{1.1}, we shall suppose that
 the given gravitational potential function $\Phi$ fulfills
$\Phi\in W^{2,\infty}(\Omega)$ and the motility function
\begin{equation}\label{1.10}
\phi\in C^1[0,\infty)\cap C^3(0,\infty)\, \, \hbox{with } \,    \phi(0)=0, \phi'(0)>0\,  \,   \hbox {and}  \,  \,  \phi>0\,  \,   \, \hbox{on}\,(0,\infty).
 \end{equation}

Our main result asserts the global solvability of \eqref{1.1} \eqref{1.7} \eqref{1.8}. A natural next step might comprise describing the large-time behavior
of the solutions in the absence of a logistic source.
For example, the results from \cite{WinklerAIHP}
 suggest that  for the case of  $\mu=0$, the solution component $n$ of  \eqref{1.1} \eqref{1.7} \eqref{1.8}
 might stabilize toward inhomogeneous states in the large
time limit.

Alas, up to now, in contrast to the availablity of results on large-time behavior of solutions to the Keller-Segel type chemotaxis system in the literature,
corresponding results for the chemotaxis system involving signal-dependent degeneracy appear scant and available only in some particular cases under smallness conditions on the initial data even in the one-dimensional setting (\cite{WinklerAIHP,PL,WinklerTAMS1}). The difficulty in the latter situation is that the degeneracy of the signal-dependent motility
reduces the a priori regularity of the solutions; as a consequence, investigations of asymptotic behavior via Lyapunov functions are no longer accessible.
In the chemotaxis-fluid context, the situation is further complicated by the fact that the estimates involving $u$ have unfavorable time dependence.
Thus, a full treatment of the large-time behavior seems to require significant further
efforts and hence goes beyond the scope of the present study. We focus our attention on boundedness instead.

Our main result then reads as follows:
\begin{theorem}
	Let $\Omega\subset\mathbb{R}^2$ be a bounded domain with smooth boundary and $\kappa=1$.
Assume that the motility function $\phi$ fulfills \eqref{1.10} and the initial data satisfy
\eqref{1.9}.

(i) If $\mu=0$, then for all $K>0$ there is $\delta(K)>0$ with  the property that
 whenever $\|v_{0}\|_{L^{\infty}(\Omega)}<K$ and
	\begin{equation}\label{1.11}
		\int_{\Omega}n_0< \delta(K),
	\end{equation}
  \eqref{1.1} \eqref{1.7} \eqref{1.8}  possesses a global classical solution;

(ii) If $\mu>0$, then \eqref{1.1} \eqref{1.7} \eqref{1.8}   possesses a global classical solution, which is bounded in the sense that for all $p>1$, there exists some $C(p)>0$ such that
 \begin{equation}\label{1.12}
	\sup_{t>0}\left\{\|n(\cdot,t)\|_{L^p(\Omega)}\!
+\!\|v(\cdot,t)\|_{W^{1,\infty}(\Omega)}+\|A^\beta u(\cdot,t)\|_{L^2(\Omega)}\right\}\leq C(p).
\end{equation}
 Furthermore, for all $K>0$ there is $\delta(K)>0$ with  the property that
 whenever $\|v_{0}\|_{L^{\infty}(\Omega)}<K$ and
	\begin{equation}\label{1.13}
		\int_{\Omega} v_0< \delta(K),
	\end{equation}
then we have $$
\displaystyle\lim_{ \overline{t\rightarrow \infty}} \|n(\cdot,t)-1\|_{L^2(\Omega)}= 0,\quad
\displaystyle\lim_{t\rightarrow \infty} \|v(\cdot,t)\|_{W^{1,\infty}(\Omega)} =0, $$
as well as
$$
\displaystyle\lim_{ \overline{t\rightarrow \infty}} \|u(\cdot,t)\|_{W^{1,2}(\Omega)}=0.
$$
\end{theorem}
\vskip1mm

It is noted that when the motility function $\phi$ in \eqref{1.1} satisfies
$\phi\in  C^3[0,\infty), \phi>0$ on $[0,\infty)$, the global classical solvability and uniform boundedness of \eqref{1.1} \eqref{1.7} \eqref{1.8} have been proven in
\cite{Winklerpreprint} (cf their Theorem 1.2).  In fact,
due to the positivity of $\phi$, $L\log L$ bounds of $n$ can be derived by means of  a testing procedure.
These bounds can then establish the $L^p$ boundedness of $u$ for any $p>2$ when combined with results asserted in Theorem 1.2 of \cite{Winklerpreprint}.
However, in our situation, unlike in the fluid-free counterpart of (1.1), even \eqref{1.6} cannot provide us substantial a priori regularity information for
   $ \int^T_0\int_{\Omega} n^2 \phi(v)$.

The core of our consideration in the case of $\mu=0$  consists of
   tracking the evolution of the functional
  \begin{equation}\label{1.14}
  \mathcal{F}_1(t):=\int_\Omega n\ln n+ \int_\Omega  \frac{|\nabla v|^2}v+C\int_\Omega |u|^2
\end{equation}
 with some $C>0$, and identifying suitable quasi-entropy properties that it has, provided that the mass of  $n_0$ is suitably small
 (Lemma \ref{lemma3.5}). The properties \eqref{3.16} \eqref{3.17} obtained in Lemma \ref{lemma3.5}
will then form the basis for the derivation of further a priori estimates, such as the time-dependent boundedness of
$\|n(\cdot,t)\|_{L^p(\Omega)} $ (Lemma \ref{lemma3.6}), $\|A^\beta u(\cdot,t)\|_{L^2(\Omega)}$ (Lemma \ref{lemma3.7}), and inter alia
 the time-dependent lower bound for $v$ (Lemma \ref{lemma3.6}). The latter will allow us to achieve local-in-time $L^\infty$ bounds for $n$ through a Moser-type iterative argument (Lemma \ref{lemma3.10}).

To investigate the large-time behavior of classical solutions in the case $\mu>0$,  it is required  to  derive  the uniform boundedness of global classical solutions of  \eqref{1.1} \eqref{1.7} \eqref{1.8}. The  crucial step is  to verify that
 $$
\mathcal{F}_2(t):=   \int_\Omega n\ln n+ C_1 \int_\Omega
\frac{|\nabla v|^2}v
+C_1 \int_\Omega |u|^2
$$
satisfies
  $$
\mathcal{F}_2'(t)+\frac{1}{C_2} \mathcal{F}_2(t)+ \frac{1}{C_2}
 \int_\Omega \frac{|\nabla v|^4}{v^3}\leq C_2,
$$
where $C_i=C_i(\|v_{0}\|_{L^{\infty}(\Omega)})>0$,  $i=1,2$, by taking advantage of the damping effect of the logistic term (Lemma 4.3). Indeed,
making use of the logistic term in the first equation of (1.1),
Lemma 4.3 provides a priori time-independent estimates for
$
 \int^{t}_{(t-1)_+} \int_\Omega  \frac{|\nabla v|^4}{v^3},
$
which, by means of a standard testing procedure and using Lemmas \ref{lemma2.5} and \ref{lemma4.4}, in turn lead us to the time-uniform boundedness
of  $\|n(\cdot,t)\|_{L^p(\Omega)}$. Thereafter, the crucial work is to establish lower bounds of  $\|n(\cdot,t)\|_{L^1(\Omega)}$
through contolling both the destabilizing action of cross-diffusion and the population-diminishing effect of the quadratic death term in (1.1).
However, in the context of the degeneracy of the motility function, the sparse regularity property of $n$ is only sufficient to consider the time average value of
$\|n(\cdot,t)\|_{L^1(\Omega)}$ rather than $\|n(\cdot,t)\|_{L^1(\Omega)}$ itself. Indeed, with the help of the boundedness of $\int_\Omega \frac{|\nabla v(\cdot,t)|^4}{v^3(\cdot,t)}$
established in Lemma 4.6,  a lower bound of the time average value of $\|n(\cdot,t)\|_{L^1(\Omega)}$ is derived in Lemma 4.7, provided that $\|v_0\|_{L^1(\Omega)}$ is sufficiently small. Thereafter, thanks to the decreasing property of $\|v(\cdot,t)\|_{L^1(\Omega)}$, this lower bound allows us to derive the decay property of $\|v(\cdot,t)\|_{W^{1,\infty}(\Omega)}$ (Lemma \ref{lemma4.8}). The latter becomes the starting point of the derivation of the large-time behaviour of the solution components $n$ and $u$ (Lemma \ref{lemma4.9} and Lemma \ref{lemma4.10}).

\section{Preliminaries}

Let us recall  a standard approach to the local solvability for chemotaxis-fluid problems \eqref{1.1} \eqref{1.7} \eqref{1.8}, which will form the basis for establishing global solvability using
an extensibility argument.
\begin{lemma}\label{lemma2.1}(Local existence).
	Let $\Omega\subset\mathbb{R}^{2}$ be a bounded domain with smooth boundary, and $\Phi\in W^{2,\infty}(\Omega)$. Assume the motility function $\phi$  satisfies \eqref{1.10} and the initial data $(n_{0},v_{0},u_{0})$
	satisfy \eqref{1.9}. Then there exist $T_{max}\in(0,\infty]$ and
functions
 \begin{equation}\label{2.1}
		\left\{\begin{array}{l}
			n\in C^{0}(\bar{\Omega}\times[0,T_{max}))\cap C^{2,1}(\bar{\Omega}\times(0,T_{max})),\\
			v\in C^{0}(\bar{\Omega}\times[0,T_{max}))\cap C^{2,1}(\bar{\Omega}\times(0,T_{max})),\\
			u\in C^{0}(\bar{\Omega}\times[0,T_{max}); \mathbb{R}^2)\cap C^{2,1}(\bar{\Omega}\times(0,T_{max}); \mathbb{R}^2)\,\, \hbox{and}\\
P\in C^{1,0}(\bar{\Omega}\times (0,T_{max})),
		\end{array}
		\right.
	\end{equation}
such that  $n> 0$ and  $v> 0$ in $\overline\Omega\times (0,T_{max})$, that  \eqref{1.1} \eqref{1.7} \eqref{1.8} is solved classically in $\Omega\times (0,T_{max})$,  	and that  if
 $T_{max}<\infty$, then for $q>2$ and $\beta\in (\frac12,1)$,
	\begin{equation}\label{2.2}
		\limsup\limits_{t\nearrow T_{max}}\left\{ \|n(\cdot,t)\|_{L^{\infty}(\Omega)}\!
+\!\|v(\cdot,t)\|_{W^{1,q}(\Omega)}+\|A^\beta u(\cdot,t)\|_{L^2(\Omega)}\right\}=\infty.
	\end{equation}
Moreover, in the case of $\mu=0$,
\begin{equation}\label{2.3}
\int_\Omega  n(\cdot,t)=\int_\Omega  n_0 \quad\hbox{for all}  \,\, t
\in(0,T_{max})
\end{equation}
and in the case of $\mu>0$,
\begin{equation}\label{2.3b}
\int_\Omega  n(\cdot,t)\leq \max\{\int_\Omega  n_0,|\Omega|\} \quad\hbox{for all}  \,\, t
\in(0,T_{max})
\end{equation}
and for  all $\mu\geq 0$,  \begin{equation} \label{2.4}
 \|v(\cdot,t)\|_{L^\infty(\Omega)}\leq K:=\|v_{0}\|_{L^{\infty}(\Omega)}
 \quad\hbox{for all}  \,\, t\in(0,T_{max}).
  \end{equation}
\end{lemma}
\begin{proof}[Proof]
This can  be  proved by  appropriate modifications of the standard arguments from local existence theories of taxis-type parabolic problems involving nonlinear diffusion (see \cite{WinklerNARWA,WinklerCPDE,HYJin}  for detailed proofs addressing two closely related situations).  Indeed, in view of the assumptions \eqref{1.9} and \eqref{1.10},  according to well-established methods based on the contraction mapping principle  along
with standard parabolic regularity theory (\cite{Amann}), one can conclude that with some $T_{max}\in (0,\infty]$,  \eqref{1.1} \eqref{1.7} \eqref{1.8} admits a unique quadruple of functions $ (n, v, u, P)$, up to addition of constants to  the pressure $P$, fulfilling (2.1) as a solution on
$(0, T_{max})$, and that either $T_{max}=\infty$, or $T_{max}<\infty$ in which case for all $q>2$ and $\beta\in (\frac12,1)$
	\begin{equation}\label{2.6}
		\limsup\limits_{t\nearrow T_{max}}\left\{ \|n(\cdot,t)\|_{L^{\infty}(\Omega)}\!
+\!\|v(\cdot,t)\|_{W^{1,q}(\Omega)}+\|A^\beta u(\cdot,t)\|_{L^2(\Omega)}
+ \|\frac1{v(\cdot,t)}\|_{L^{\infty}\Omega)}\right\}=\infty.
	\end{equation}

In order to verify the validity of criterion  \eqref{2.2}, we presuppose that
$T_{max}<\infty$, but for some $q>2$ and $\beta\in (\frac12,1)$,
	\begin{equation}\label{2.7b}
	\|n(\cdot,t)\|_{L^{\infty}(\Omega)}\!
+\!\|v(\cdot,t)\|_{W^{1,q}(\Omega)}+\|A^\beta u(\cdot,t)\|_{L^2(\Omega)}\leq c_1
	\end{equation}	
with some $c_1>0$ for all $t<T_{max}$. From the second equation in (1.1), it follows that $v_t-u\cdot \bigtriangledown v\geq \bigtriangleup v -c_1 v$, and then by the parabolic comparison principle,  we have
 that
$v(x,t)\geq \inf_{x\in\Omega} v_0(x) e^{- c_1 t}$ for all $x\in \Omega$ and $t<T_{max}$, which particularly shows that $v(x,t)\geq c_2:=\inf_{x\in\Omega} v_0(x) e^{- c_1 T_{max}}$ (see the proof of either Lemma 3.7 of \cite{Fuest} or Lemma 2.1 of \cite{WinklerNARWA}  for the details thereof), which along with \eqref{2.7b} would contradict \eqref{2.6} and thereby \eqref{2.2} is actually valid.
Furthermore, interpreting  the first equation in (1.1) as the inhomogeneous linear equation
$$
n_t=\varphi(v) \Delta n+a(x,t)\cdot \nabla n + b(x,t) n
$$
with the smooth functions given by
$a(x,t):= 2\varphi'(n)\nabla v+u$ and    $b(x,t):=\varphi'(v) \Delta\varphi +\varphi''(v)|\nabla v|^2+\mu-\mu n$, $n $ is actually strictly positive in $\overline \Omega\times (0, T_{max})$  by the strong maximum principle.  \eqref{2.3} and \eqref{2.3b} result from an integration of the first equation with respect to $x\in\Omega$ and using the solenoidality of $u$.  The inequality \eqref {2.4} follows from
an application of the comparison principle to the second equation in (1.1) immediately.
\end{proof}

\begin{lemma}\label{lemma2.2}Let  $\phi$  satisfy \eqref{1.10} and $ \|v_{0}\|_{L^{\infty}(\Omega)}\leq K$. Then there exist $\lambda(K)>0$
 and $\Lambda(K)>0$ such that $\lambda(K)v\leq \phi(v)\leq \Lambda(K)v $ and $|\phi'(v)|\leq \Lambda(K)$ in $\Omega\times (0,T_{max})$.
\end{lemma}
\begin{proof} Utilizing the fact that  $\|v(\cdot,t)\|_{L^\infty(\Omega)}\leq K$ for all $ t\in(0,T_{max})$ obtained in Lemma \ref{lemma2.1} and the assumption  \eqref{1.10}, one can readily complete the proof of this lemma (see Lemma 3.2 of \cite{WinklerAIHP} for details).
\end{proof}
 The following lemmas  play an important role in the derivation of the associated  functional inequality, through which  the associated bounds of solutions are established. 

\begin{lemma}\label{lemma2.3}
(\cite [Lemma 2.2 and Lemma 2.3] {Winkler(2020)})
	Suppose that $\Omega \subset\mathbb{R}^{2}$ is a smooth bounded domain. Then
for all $\varepsilon>0$
there exists $M=M (\varepsilon, \Omega)>0$ such that if $0 \not\equiv \varphi \in C^{0}(\overline\Omega )$ is nonnegative and $\psi\in W^{1,2}(\Omega)$, then for each $a>0$,
\begin{equation}\label{2.7}
		\begin{split}
			\int_{\Omega}\varphi | \psi | &\leq \frac{1}{a}\int_{\Omega}\varphi \ln\frac{ \varphi}{\overline{\varphi}}+ \frac{(1 + \varepsilon )a}{8\pi} \cdot \biggl\{\int_{\Omega}\varphi \biggr\} \cdot\int_{\Omega} | \nabla \psi |^{2}\\
			&\quad+M a\cdot \biggl\{\int_{\Omega} \varphi \biggr\} \cdot \biggl\{ \int_{\Omega}| \psi | \biggr\}^{2}+\frac{M}{a}\int_{\Omega}\varphi ,
		\end{split}
	\end{equation}
where $\overline{\varphi}:= \frac{1}{|\Omega |}\int_{\Omega} \varphi$.
\end{lemma}

\begin{lemma}\label{lemma2.4}(\cite[Lemma 3.5]{WinklerAIHP})
	Suppose that $\Omega \subset\mathbb{R}^{2}$ is a smooth bounded domain. For  all $p\geq 1$, there is $\Gamma >1$ such that
 for each  $\varphi \in C^{1}(\overline\Omega )$ and $\psi\in C^{1}(\overline\Omega)$  fulfilling $\varphi>0$ and $\psi>0$ in $\overline\Omega$, we have
\begin{equation}\label{2.9}
		\begin{split}
			\int_{\Omega} \frac{\varphi^p}{\psi} |\nabla \psi|^2 &\leq \eta \int_{\Omega}\varphi^{p-2}\psi|\nabla \varphi|^2+
\eta \int_{\Omega} \varphi \psi+p^2\Gamma^p(1+\frac 1 \eta)
 \biggl\{
 \int_{\Omega}\varphi^p+
 \biggl\{ \int_\Omega \varphi \biggr\} ^{2p-1} \biggr\} \cdot
 \int_{\Omega} \frac{|\nabla \psi|^4}{\psi^3}	
 	\end{split}
	\end{equation}
for all $\eta>0$.
\end{lemma}

\begin{lemma}\label{lemma2.5}
(\cite [Lemma 2.3] {PWang})
Let $T\in (0,\infty], 0<\tau<T$ and suppose that $y$ is a nonnegative absolutely continuous function satisfying
\begin{equation}\label{2.11}
y'(t)+a(t)y(t)\leq b(t)y(t)+c(t) \quad ~~~~~\hbox{for a.e. }t\in (0,T)
\end{equation}
for some functions $a,b,c\in L^1_{loc}(0,T)$ where $a(t)>0,b(t)\geq 0, c(t)\geq 0$ and such that there exist
$
b_1,c_1>0$ and $\varrho>0$ such that
$$
\displaystyle\sup_{0\leq t\leq T} \int^{t}_{(t-\tau)_+} b(s)ds\leq b_1,~~
\displaystyle\sup_{0\leq t\leq T} \int^{t}_{(t-\tau)_+} c(s)ds\leq c_1
$$
 and
$$
\int^{t}_{(t-\tau)_+} a(s)ds-\int^{t}_{(t-\tau)_+} b(s)ds\geq \varrho
~~~\hbox{for any}~ t\in (0,T).$$
Then
$$
y(t)\leq y(0)e^{b_1}+\displaystyle\frac{c_1 e^{2b_1}}{1-e^{-\varrho}}+c_1e^{b_1}~~\hbox{ for all}~t\in (0,T).
$$
\end{lemma}

\section{Proof of  Theorem 1.1(i)}

In this section, we establish the global classical solvability of the system \eqref{1.1} \eqref{1.7} \eqref{1.8}  with $\mu=0$.
As a starting point of the analysis, we derive a basic differential inequality. Similar ideas can be found in \cite{WinklerCPDE}.
\begin{lemma}\label{lemma3.1} Let  $ \|v_{0}\|_{L^{\infty}(\Omega)}\leq K$ and suppose that
$T_{max}<\infty$. Then for all $t<T_{max}$, we have
	\begin{equation}\label{3.1}
\frac d{dt}\int_\Omega n\ln n+
\frac{\lambda(K)} 2\int_\Omega\frac{v}{n}|\nabla n|^2
\leq \frac{\Lambda^2(K)}{2\lambda(K)}
 \int_\Omega\frac nv |\nabla v|^2,
\end{equation}
where $\lambda(K)$ and $\Lambda(K)$ are as in Lemma \ref{lemma2.2}.

\end{lemma}

\begin{proof} According to Lemma  \ref{lemma2.1}, we have $n(x,t)>0$ for $x\in \Omega,t>0$.  Multiplying the first equation in (\ref{1.1}) by $\frac1{n}$,  using the solenoidality of $u$ and by the Young inequality,
we arrive at
\begin{equation}\label{3.2}
\begin{split}
\frac d{dt}\int_\Omega n\ln n
&=\int_\Omega\ln n
(\Delta (n\phi(v))-u\cdot\nabla n)\\
&=-\int_\Omega\frac{\nabla n}{n}
(\phi(v) \nabla n+\phi'(v)n \nabla v )\\
&\leq -\int_\Omega \frac{\phi(v)} {2n}|\nabla n|^2+
\int_\Omega \frac{n|\phi'(v)|^2} {2\phi(v)} |\nabla v|^2\\
&\leq- \frac{\lambda(K)} 2\int_\Omega\frac{v}{n}|\nabla n|^2
+
\frac{\Lambda^2(K)}{2\lambda(K)}\int_\Omega \frac nv |\nabla v|^2,
\end{split}
\end{equation}
where we have applied Lemma \ref{lemma2.2} to obtain the estimates
$$
\phi(v)\geq \lambda(K)v, \quad \frac{|\phi'(v)|^2} {\phi(v)}\leq \frac{\Lambda^2(K)}{\lambda(K)v}\quad \hbox{in}\,
\,\Omega\times (0,T_{max}).
$$
\end{proof}

Now to further investigate the integral on the right of \eqref{3.1}, we shall trace the evolution of
$\int_\Omega  \frac{|\nabla v|^2}v$, with the aim
of leveraging on favourable singularly weighted dissipation rates arising there. We shall also apply the same to similar quantities such as $\int_\Omega \frac {|\nabla v|^4}{v^3}$.
 \begin{lemma}\label{lemma3.2}
 For all $\varepsilon>0$, we have
 \begin{equation}\label{3.3}
\begin{split}
&\frac d{dt} \int_\Omega  \frac{|\nabla v|^2}v
+(\frac18-3\varepsilon)\int_\Omega \frac{|\nabla v|^4}{v^3}\\
\leq &
\frac{\lambda(K)} 4\int_\Omega\frac{v}{n}|\nabla n|^2
+
\frac4{\lambda(K)} \int_\Omega\frac nv |\nabla v|^2 +\frac1{\varepsilon}
\int_\Omega v |\nabla u|^2+C_\varepsilon\int_{\Omega}v,
\end{split}
\end{equation}
where the constant $C_{\varepsilon}>0$ depends only on $\varepsilon$.
\end{lemma}
\begin{proof}[Proof] The gist of the proof goes as in \cite[Lemmas 3.2--3.4]{WinklerCPDE}. Since it is a cornerstone of subsequent a priori estimates,
let us recall the main idea. Using integration by parts in a straightforward manner, from the second equation in (1.1)
we obtain the identity
	\begin{equation}\label{3.4}
\begin{split}
\frac d{dt}\int_\Omega  \frac{|\nabla v|^2}v
=&2 \int_\Omega \frac{\nabla v}v\cdot \nabla v_{t}-\int_\Omega\frac{|\nabla v|^2}{v^2}v_{t}\\
=&-2 \int_\Omega v|D^2\ln v|^2
+\int_{\partial\Omega}v^{-1}\frac {\partial |\nabla v|^2}{\partial \nu}
-
\int_\Omega \frac nv |\nabla v|^2- 2\int_\Omega \nabla n\cdot\nabla v
\\
&+2\int_\Omega (u\cdot \nabla v)\frac {\triangle v}v- \int_\Omega  (u\cdot \nabla v)
\frac {|\nabla v|^2}{v^2}
\\
=& -2 \int_\Omega v|D^2\ln v|^2
+\int_{\partial\Omega}v^{-1}\frac {\partial |\nabla v|^2}{\partial \nu}
-
\int_\Omega \frac nv |\nabla v|^2
- 2\int_\Omega \nabla n\cdot\nabla v\\
&+
 2\int_\Omega (u\cdot\nabla\sqrt{v}) \Delta \sqrt{v}\\
 =& -2 \int_\Omega v|D^2\ln v|^2
+\int_{\partial\Omega}v^{-1}\frac {\partial |\nabla v|^2}{\partial \nu}
-
\int_\Omega \frac nv |\nabla v|^2
- 2\int_\Omega \nabla n\cdot\nabla v\\
&+
 2
\int_\Omega (\nabla\sqrt{v}\otimes\nabla \sqrt{v}):\nabla u.
\end{split}
\end{equation}
 Here the notion $A : B$ denotes $Tr(AB) = A_{ij}B_{ji}$ for two $2\times2$ matrices $A, B$.

Thanks to Lemma 2.4 of \cite{Jiang}, one can conclude that for any $\varepsilon>0$, there is $C_{\varepsilon}>0$ such that
\begin{equation}\label{3.5}
\int_{\partial\Omega}v^{-1}\frac {\partial |\nabla v|^2}{\partial \nu} \leq \varepsilon\int_\Omega v|\Delta\ln v|^2+\varepsilon
\int_\Omega \frac{|\nabla v|^4}{v^3}+C_\varepsilon\int_{\Omega}v.
\end{equation}
Combining  \eqref{3.5} with   \eqref{3.4} then yields
\begin{align}\label{3.6}
&\frac d{dt} \int_\Omega  \frac{|\nabla v|^2}v
+2\int_\Omega v|D^2\ln v|^2+ \int_\Omega \frac nv |\nabla v|^2\notag\\
= & -2\int_\Omega \nabla n\cdot\nabla v+2
\int_\Omega (\nabla\sqrt{v}\otimes\nabla \sqrt{v}):\nabla u+\int_{\partial\Omega}v^{-1}\frac {\partial |\nabla v|^2}{\partial \nu}\\
\leq & -2\int_\Omega \nabla n\cdot\nabla v+2
\int_\Omega \frac{|\nabla v|^2}{v}|\nabla u|+\int_{\partial\Omega}v^{-1}\frac {\partial |\nabla v|^2}{\partial \nu}\notag\\
\leq & -2\int_\Omega \nabla n\cdot\nabla v+  \frac1{\varepsilon}
\int_\Omega v |\nabla u|^2+\varepsilon\int_\Omega v|\Delta\ln v|^2+
2\varepsilon
\int_\Omega \frac{|\nabla v|^4}{v^3}+C_\varepsilon\int_{\Omega}v
\notag\\
\leq &
\frac{\lambda(K)} 4\int_\Omega\frac{v}{n}|\nabla n|^2
+
\frac4{\lambda(K)} \int_\Omega\frac nv |\nabla v|^2 +\frac1{\varepsilon}
\int_\Omega v |\nabla u|^2+\varepsilon\int_\Omega v|\Delta\ln v|^2+
2\varepsilon
\int_\Omega \frac{|\nabla v|^4}{v^3}+C_\varepsilon\int_{\Omega}v.\notag
\end{align}
Since $ |\triangle v|^2\leq 2|D^2 v|^2$ by the Cauchy-Schwarz inequality, and thanks to Lemma 3.4 of \cite{WinklerNARWA} or Lemma 3.3 of \cite{WinklerDCDSB(2022)}, we have \begin{equation}\label{3.7}
\frac 1 {2(7 + 4\sqrt{2})}\int_\Omega\frac {|\triangle v|^2}{v}\leq
\frac 1 {7 + 4\sqrt{2}}
\int_\Omega\frac{|D^2 v|^2}{v}\leq 2 \int_\Omega v|D^2\ln v|^2
\end{equation}
and \begin{equation}\label{3.8}
\int_\Omega v^{-3}|\nabla v|^4\leq (2+\sqrt{2})^2 \int_\Omega v|D^2\ln v|^2
\end{equation}
for all $t\in (0, T_{max})$.
Therefore from \eqref{3.6}--\eqref{3.8}, we can see  that
\begin{equation}\label{3.9}
\begin{split}
&\frac d{dt} \int_\Omega  \frac{|\nabla v|^2}v
+(\frac18-3\varepsilon)\int_\Omega \frac{|\nabla v|^4}{v^3}
+ \int_\Omega \frac nv |\nabla v|^2\\
\leq &
\frac{\lambda(K)} 4\int_\Omega\frac{v}{n}|\nabla n|^2
+
\frac4{\lambda(K)} \int_\Omega\frac nv |\nabla v|^2 +\frac1{\varepsilon}
\int_\Omega v |\nabla u|^2+C_\varepsilon\int_{\Omega}v,
\end{split}
\end{equation}
and thus complete the proof of this lemma.
\end{proof}

Now an appropriate combination of the previous two lemmas with Lemma \ref{lemma2.3} and Lemma \ref{lemma2.4} reveals a
quasi-energy feature if the mass of $n_0$ is sufficient small. This is a crucial observation that enables the subsequent analysis.
\begin{lemma}\label{lemma3.3}
	Let $(n,v,u,P)$ be the  classical solution of \eqref{1.1} \eqref{1.7} \eqref{1.8} obtained in Lemma \ref{lemma2.1}.
Then for any $K>0$, one can find $\delta_1(K)>0 $  with the property that  if  $\|v_{0}\|_{L^{\infty}(\Omega)}\leq K$ and  $\int_\Omega n_0<\delta_1(K)$, there exists $C>0$  independent of $K$ such that
 \begin{equation}\label{3.10}
\begin{split}
&\frac d{dt} (\int_\Omega n\ln n+\int_\Omega  \frac{|\nabla v|^2}v)+
\frac{\lambda(K)} 8\int_\Omega\frac{v}{n}|\nabla n|^2+
\frac1{48}
\int_\Omega \frac{|\nabla v|^4}{v^3}
\\
\leq &
C\int_{\Omega}v+ C\int_\Omega v |\nabla u|^2
+\frac{\lambda(K)}8\int_\Omega nv
\end{split}
\end{equation}
for all $t\in(0,T_{max})$.
 \end{lemma}
\begin{proof}[Proof]
 The simple combination of Lemma \ref{lemma3.1} with  Lemma \ref{lemma3.2} leads to
\begin{equation}\label{3.11}
\begin{split}
&\frac d{dt} (\int_\Omega n\ln n+\int_\Omega  \frac{|\nabla v|^2}v)+
\frac{\lambda(K)} 4\int_\Omega\frac{v}{n}|\nabla n|^2+
(\frac18-3\varepsilon)
\int_\Omega \frac{|\nabla v|^4}{v^3}
\\
\leq &
C_\varepsilon\int_{\Omega}v+  \frac 1{\varepsilon}\int_\Omega v |\nabla u|^2
+(\frac{\Lambda^2(K)}{2\lambda(K)}+\frac 4{\lambda(K)}-1) \int_\Omega \frac nv |\nabla v|^2.
\end{split}
\end{equation}
To estimate the last term on the right of \eqref{3.11}, we utilize Lemma 2.4 with $p=1$ to get
\begin{equation}\label{3.12}
\begin{split}
&(\frac{\Lambda^2(K)}{2\lambda(K)}+\frac 4{\lambda(K)}-1) \int_\Omega \frac nv |\nabla v|^2\\
\leq & \frac{\lambda(K)}8 \int_\Omega\frac{v}{n}|\nabla n|^2
+\frac{\lambda(K)}8\int_\Omega nv+
(1+\frac {4(\Lambda^2(K)+8)}{\lambda^2(K)})
\int_\Omega n \cdot
\int_\Omega \frac{|\nabla v|^4}{v^3}.
\end{split}
\end{equation}
Inserting \eqref{3.12} into \eqref{3.11}, we then obtain  that
\begin{equation}\label{3.13}
\begin{split}
&\frac d{dt} (\int_\Omega n\ln n+\int_\Omega  \frac{|\nabla v|^2}v)+
\frac{\lambda(K)} 8\int_\Omega\frac{v}{n}|\nabla n|^2+
(\frac18-3\varepsilon)
\int_\Omega \frac{|\nabla v|^4}{v^3}
\\
\leq &
C_\varepsilon\int_{\Omega}v+  \frac 1{\varepsilon}\int_\Omega v |\nabla u|^2
+\frac{\lambda(K)}8\int_\Omega nv+
C_1(K) \int_\Omega n_0 \cdot
\int_\Omega \frac{|\nabla v|^4}{v^3}
\end{split}
\end{equation}
with $C_1(K):=1+\frac {4(\Lambda^2(K)+8)}{\lambda^2(K)}$  due to $\eqref{2.3}$, and hence
\begin{equation*}
\begin{split}
&\frac d{dt} (\int_\Omega n\ln n+\int_\Omega  \frac{|\nabla v|^2}v)+
\frac{\lambda(K)} 8\int_\Omega\frac{v}{n}|\nabla n|^2+
(\frac1 {24}-3\varepsilon)
\int_\Omega \frac{|\nabla v|^4}{v^3}
\\
\leq &
C_\varepsilon\int_{\Omega}v+  \frac 1{\varepsilon}\int_\Omega v |\nabla u|^2
+\frac{\lambda(K)}8\int_\Omega nv,
\end{split}
\end{equation*}
whenever  $\int_\Omega n_0 < \delta_1(K):=\frac1{12C_1(K)}$. Upon the choice $\varepsilon:=\frac1{144}$, this readily yields \eqref{3.10} with
$C:=144+C_{\frac1{144}}$
\end{proof}

Drawing on the functional inequality \eqref{2.7}, we can now derive basic regularity features of the fluid field using a Navier-Stokes energy analysis.

\begin{lemma}\label{lemma3.4}
There exists $C=C(M)>0$ such that
\begin{equation}\label{3.14}
\frac d{dt}\int_\Omega|u|^2
+\int_\Omega|\nabla u|^2
\leq C \int_\Omega n \cdot\int_\Omega n\ln\frac n{\overline n}+C\{\int_\Omega n\}^2
\end{equation}for all $t\in (0,T_{max})$, where $M:=M(1,\Omega)$ is provided by Lemma \ref{lemma2.3}.
\end{lemma}
\begin{proof}[Proof]
We invoke the Poincar\'{e} inequality to fix $C_p>0$ fulfilling
\begin{equation*}
\int_\Omega|\varphi|^2\leq C_p\int_\Omega|\nabla\varphi|^2\mbox{ for all }\varphi\in W_0^{1,2}(\Omega,\mathbb{R}^2).
\end{equation*}
Moreover, using the Cauchy-Schwarz inequality, we have
\begin{equation}\label{3.15}
\bigg\{\int_\Omega|u|\bigg\}^2
\leq|\Omega|\int_\Omega|u|^2
\leq C_p|\Omega|\int_\Omega|\nabla u|^2.
\end{equation}
Thanks to Lemma \ref{lemma2.3}, testing the third equation in \eqref{1.1} against $u$ shows that
\begin{equation*}
\begin{split}
&\frac1 2\frac d{dt}\int_\Omega|u|^2+\int_\Omega|\nabla u|^2\\
&=
\int_\Omega nu\cdot\nabla\Phi\\
&\leq
\|\nabla\Phi\|_{L^\infty}\int_\Omega n|u|\\
&\leq
\|\nabla\Phi\|_{L^\infty}\int_\Omega n|u_1|
+\|\nabla\Phi\|_{L^\infty}\int_\Omega n|u_2|\\
&\leq
\|\nabla\Phi\|_{L^\infty}
\bigg\{\frac2{a}\int_\Omega n\ln\frac n{\overline{n}}
+\frac{am}{2\pi}\int_\Omega|\nabla u|^2
+2Mam(\int_\Omega |u|)^2+\frac{2Mm}{a}\bigg\}\\
&\leq
\frac{2\|\nabla\Phi\|_{L^\infty}}{a}
\int_\Omega n\ln\frac n{\overline{n}}
+\bigg\{\frac{am\|\nabla\Phi\|_{L^\infty}}{2\pi}
+2MamC_p|\Omega|\|\nabla\Phi\|_{L^\infty}\bigg\}
\int_\Omega|\nabla u|^2
+\frac{2Mm\|\nabla\Phi\|_{L^\infty}}{a}
\end{split}
\end{equation*}
with $m:=\int_\Omega n$.

Choosing $a=\frac1{2m\|\nabla\Phi\|_{L^\infty}(1+2MC_p|\Omega|)}$ in the above inequality,
we then have
\begin{equation*}
\frac d{dt}\int_\Omega|u|^2
+\int_\Omega|\nabla u|^2\leq 2Cm\int_\Omega n\ln\frac n{\overline{n}}+2Cm^2
\end{equation*}
with $C:=4(M+1)\|\nabla\Phi\|_{L^\infty}^2(1+2MC_p|\Omega|)$, and thereby complete the proof.
\end{proof}

At this position by means of an appropriate combination of Lemma \ref{lemma3.3} with Lemma \ref{lemma3.4}, we
show that $\mathcal{F}_1$  in \eqref{1.13} enjoys a quasi-entropy property.

\begin{lemma}\label{lemma3.5} Let $\int_\Omega n_0<\delta(K):=\min\{\delta_1(K),1\}$. Then there exists $C>0$ such that
\begin{equation}\label{3.16}
\int_\Omega n(\cdot,t)|\ln n(\cdot,t)|+\int_\Omega  \frac{|\nabla v(\cdot,t)|^2}{v(\cdot,t)}+\int_\Omega |u(\cdot,t)|^2\leq C
\end{equation}
and that
\begin{equation}\label{3.17}
\int_0^t\int_\Omega\frac{v}{n}|\nabla n|^2+
\int_0^t \int_\Omega \frac{|\nabla v|^4}{v^3}
+\int^t_0\int_\Omega |\nabla u|^2
\leq C
\end{equation} for all $t\in (0,T_{max})$.
\end{lemma}
\begin{proof}[Proof]
In view of \eqref{2.4}, it follows from Lemma \ref{lemma3.3} that
there exists $C_1>0$  such that
 \begin{equation}\label{3.18}
\begin{split}
&\frac d{dt} (\int_\Omega n\ln n+\int_\Omega  \frac{|\nabla v|^2}v)+
\frac{\lambda(K)} 8\int_\Omega\frac{v}{n}|\nabla n|^2+
\frac1{48}
\int_\Omega \frac{|\nabla v|^4}{v^3}
\\
\leq &
C_1+ C_1\int_\Omega  |\nabla u|^2
+\frac{\lambda(K)}8\int_\Omega nv.
\end{split}
\end{equation}
for all  $t<T_{max}$.

By linearly combining \eqref{3.18} and \eqref{3.14}, we see that
$$\mathcal{F}_1(t):=\int_\Omega n\ln n+ \int_\Omega  \frac{|\nabla v|^2}v+2C_1\int_\Omega |u|^2
$$
satisfies
\begin{equation}\label{3.19}
\begin{split}
&\frac {d\mathcal{F}_1}{dt} +
\frac{\lambda(K)} 8\int_\Omega\frac{v}{n}|\nabla n|^2+
\frac1{48}
\int_\Omega \frac{|\nabla v|^4}{v^3}
+
C_1\int_\Omega |\nabla u|^2
\\
\leq &  C_2\int_\Omega nv+C_2+C_2 \int_\Omega n\ln n\\
\leq &  C_3+C_2\mathcal{F}_1
\end{split}
\end{equation}
with some  $C_2>0$ and  $C_3>0$, where the fact that $x\ln x\geq -e^{-1} $ for all $x>0$ and $\int_\Omega n=\int_\Omega n_0\leq 1$ is used.

From an ODE comparison argument applied to \eqref{3.19}, it follows that
$$
\mathcal{F}_1(t)\leq \mathcal{F}_1(0)e^{C_2T_{max}} +e^{C_2T_{max}} C_3T_{max} =:C_4
$$
for all  $t\in (0,T_{max})$.
Thereupon, an integration of \eqref{3.19} further shows that
\begin{equation*}
\begin{split}
&\frac{\lambda(K)} 8\int^t_0\int_\Omega\frac{v}{n}|\nabla n|^2+
\frac1{48}
\int^t_0\int_\Omega \frac{|\nabla v|^4}{v^3}
+
C_1\int^t_0\int_\Omega |\nabla u|^2\\
\leq  & C_3T_{max}+C_2 C_4 T_{max}-\int_\Omega  n\ln n+\int_\Omega  n_0\ln n_0\\
\leq  & C_5,
\end{split}
\end{equation*}
for some  $C_5>0$, due to the fact that $x\ln x\geq -e^{-1} $ for all $x>0$, which establishes the claim of this lemma immediately.
\end{proof}
\vskip4mm
Drawing on Lemma \ref{lemma2.4} once more, we shall show that, despite the degeneracy in the $n$-equation, $n$ is locally bounded with respect to the norm in any $L^p$ space with any $p\geq 2$.

\begin{lemma}\label{lemma3.6} Let $\int_\Omega n_0<\delta(K)$. Then for $p \geq  2$, there exists  $C(p)>0$ such that
\begin{equation}\label{3.20}
\|n(\cdot,t)\|_{L^p(\Omega)} \leq C(p)
\end{equation}
for all $t\in (0, T_{max})$.
\end{lemma}
\begin{proof}
Testing the first equation in \eqref{1.1} by $n^{p-1}$, an application of Lemma \ref{lemma2.2}, the Young inequality and $\nabla\cdot u=0$ shows that
\begin{equation}\label{3.21}
\begin{split}
\frac d{dt}\int_\Omega n^p&= p\int_\Omega n^{p-1}\{ \triangle(n\phi(v))-u\cdot \nabla n\}\\
&=-p(p-1) \int_\Omega \phi(v)n^{p-2}|\nabla n|^2-
p(p-1) \int_\Omega n^{p-1} \phi'(v)\nabla n\cdot \nabla v \\
 &\leq -\frac {p(p-1)\lambda(K)}2 \int_\Omega vn^{p-2}|\nabla n|^2+\frac {p(p-1)\Lambda^2(K)}{2\lambda(K)}
 \int_\Omega \frac{n^p}v |\nabla v|^2.
 \end{split}
\end{equation}
Now thanks to  Lemma \ref{lemma2.4} once more, there exists $C_1(p)>0$ such that

\begin{align}\label{3.22}
&\frac {p(p-1)\Lambda^2(K)}{2\lambda(K)} \int_\Omega \frac{n^p}v |\nabla v|^2\\
 \leq &\frac {p(p-1)\lambda(K)}4 \left\{\int_\Omega vn^{p-2}|\nabla n|^2+ \int_\Omega n v\right\}+
 C_1(p)\left\{\int_\Omega n^p+\left(\int_\Omega  n\right)^{2p-1} \right\}\cdot  \int_\Omega \frac{|\nabla v|^4}{v^3}\notag
\end{align}
for all $t\in (0, T_{max})$.

Hence from  \eqref{3.21} and   \eqref{3.22}, we infer that  if we let $\mathcal{F}(t):=\int_\Omega n(\cdot,t)^p$, then
\begin{equation*}
\begin{split}
&\mathcal{F}'(t)+\frac {p(p-1)\lambda(K)}4 \int_\Omega vn^{p-2}|\nabla n|^2\\
\leq &C_1(p) \int_\Omega \frac{|\nabla v|^4}{v^3} \cdot \mathcal{F}(t)+
\frac {p(p-1)\lambda(K)}4 \int_\Omega n v+C_1(p)\left(\int_\Omega  n\right)^{2p-1}
 \cdot  \int_\Omega \frac{|\nabla v|^4}{v^3}
   \end{split}
\end{equation*}
and thus
\begin{equation*}
\begin{split}
\mathcal{F}(t)\leq & \left\{\mathcal{F}(0)+
\frac {p(p-1)\lambda(K)}4 \int_0^t\int_\Omega n v+C_1(p)\left(\int_\Omega  n_0\right)^{2p-1} \int_0^t\int_\Omega \frac{|\nabla v|^4}{v^3}\right\} \cdot e^{C_1(p)\int_0^t \int_\Omega \frac{|\nabla v|^4}{v^3} } \\
\leq &
\left\{\mathcal{F}(0)+
\frac {p(p-1)\lambda(K)}4 \int_\Omega  v_0+C_1(p)\left(\int_\Omega  n_0\right)^{2p-1} \int_0^t\int_\Omega \frac{|\nabla v|^4}{v^3}\right\} \cdot e^{C_1(p)\int_0^t \int_\Omega \frac{|\nabla v|^4}{v^3} }
\\
\leq & C(p),
 \end{split}
\end{equation*}
where we have used \eqref{3.17}.
\vspace{12pt}
\end{proof}

In order to invoke the extensibility criterion \eqref{2.2}, further regularity features of the solution $(n,v,u)$ are required.
In this context,  \eqref{3.20} allows us to achieve the bound of $\|A^\beta u(\cdot,t)\|_{L^2(\Omega)}$ with $\beta\in (\frac 12,1)$
by combining standard smoothing features of the Dirichlet Stokes semigroup with the boundedness of $\mathcal{P}$
on $L^q(\Omega; \mathbb{R}^2)$ for $q>1$ \cite{FM}, as done in \cite{Fuest}. Specially, we have
\begin{lemma}\label{lemma3.7} Suppose that  $\|n(\cdot,t)\|_{L^p(\Omega)} \leq C_1$ with $p>2$ for $t<T_{max}$. Then for any $\beta\in (\frac 12,1)$, there exists  $C_2=C_2(C_1,\beta)>0$ such that
\begin{equation}\label{3.23}
\|A^\beta u(\cdot,t)\|_{L^2(\Omega)} \leq C_2
\end{equation}
and
\begin{equation}\label{3.24}
\sup_{t\in (0,T_{max})}\|u(\cdot,t)\|_{L^\infty(\Omega)} \leq C_2
\end{equation}
for all $t\in (0, T_{max})$.
\end{lemma}
\begin{proof} From the assumption that $\|n(\cdot,t)\|_{L^p(\Omega)} \leq C_1$ for some $p>2$, it follows that for
all $q>1$,  $\|u(\cdot,t)\|_{L^q(\Omega)} \leq C_3(q)$ with $C_3(q)>0$, thanks to Theorem 1.2 in \cite{Fuest}.  Furthermore,
according to standard smoothing features of the Dirichlet Stokes semigroup \cite[p.~201]{Friedman}, one can find $\lambda>0$ and constants
$C_i>0$ $(i=4,\ldots,7)$ such that
\begin{align*}
&\|A^\beta u(\cdot,t)\|_{L^2(\Omega)}\notag\\
\leq & C_4
\|A^\beta  u_0(\cdot)\|_{L^2(\Omega)}
+\int_{0}^t
\|A^\beta e^{-(t-s)A}
\mathcal{P}[n\nabla\Phi]\|_{L^2(\Omega)}ds
+\int_{0}^t
\|A^\beta e^{(t-s)A}
\mathcal{P}[(u\cdot\nabla)u]\|_{L^2(\Omega)}ds\\
\leq &C_4
\|A^\beta  u_0(\cdot)\|_{L^2(\Omega)}
+C_5\int_{0}^t (t-s)^{-\beta}
e^{-\lambda(t-s)}
\|n\nabla\Phi\|_{L^2(\Omega)}ds \notag\\
&\quad+C_5\int_{0}^t(t-s)^{-\beta-\frac 1p+\frac 12}
e^{-\lambda(t-s)}
\|\mathcal{P}[(u\cdot\nabla)u]\|_{L^p(\Omega)}ds \notag\\
\leq & C_6
+C_6\int_{0}^t(t-s)^{-\beta}
e^{-\lambda(t-s)}\|n\|_{L^2(\Omega)}ds +C_6\int_{0}^t(t-s)^{-\beta-\frac 1p+\frac 12}
e^{-\lambda(t-s)}\|u\|_{L^{\frac{2p}{2-p}}(\Omega)}\|\nabla u\|_{L^2(\Omega)}ds\\
\leq & C_7+C_7 \int_{0}^t(t-s)^{-\beta-\frac 1p+\frac 12}
e^{-\lambda(t-s)}\|\nabla u\|_{L^2(\Omega)}ds . \notag
\end{align*}
 On the other hand, by means of the
interpolation inequality (\cite[Theorem 14.1]{Friedman}), one can conclude that there exist $\theta=\theta(\beta)<1$ and $C_8=C_8(\beta)>0$ such that
$$\|\nabla u\|_{L^2(\Omega)}\leq C_8 \|A^\beta u(\cdot,t)\|^{\theta}_{L^2(\Omega)}  \|u(\cdot,t)\|^{1-\theta}_{L^2(\Omega)}.
$$
Therefore for any given $T<T_{max}$,  writing $I(T):=\displaystyle\sup_{0<s<T}\{ \|A^\beta u(\cdot,s)\|_{L^2(\Omega)}\}$, we get
\begin{align}
\|A^\beta u(\cdot,t)\|_{L^2(\Omega)}
\leq & C_7+C_7C_8C_3^{1-\theta} I^{\theta}(T)\int_{0}^t(t-s)^{-\beta-\frac 1p+\frac 12}
e^{-\lambda(t-s)} ds\notag \\
\leq &  C_7+C_9 I^{\theta}(T)
\end{align}
with $C_9:=C_7C_8C_3^{1-\theta} \int_{0}^{\infty}\sigma^{-\beta-\frac 1p+\frac 12}e^{-\lambda \sigma} d\sigma$,
which establishes \eqref{3.23} with $C_2=\max\{1, (C_7+C_9)^{\frac1{1-\theta}} \}$, whereafter \eqref{3.24} results from
the imbedding result $D(A^\beta) \hookrightarrow L^\infty(\Omega;\mathbb{R}^2)$ for all $\beta\in (\frac 12,1)$.
\end{proof}

A straightforward application of the $L^p-L^q$ estimates associated with the Neumann heat semigroup can turn the regularity information in \eqref{3.24} and \eqref{3.20}  into the boundness of $\|\nabla v(\cdot,t)\|_{L^{\infty}(\Omega)}$  for all $t\in (0, T_{max})$.

\begin{lemma}\label{lemma3.8} Let $\int_\Omega n_0<\delta(K)$. Then there exists  $C>0$  such that
\begin{equation}\label{3.26}
\|v(\cdot,t)\|_{W^{1,\infty}(\Omega)}\leq C
\end{equation}
for all $t\in (0, T_{max})$.
\end{lemma}\begin{proof} Utilizing Neumann heat semigroup smoothness estimates, one can prove the inequality \eqref{3.26} in a fairly standard fashion. We refer the interested reader to the proof of \cite[Lemma 3.9] {Winklerpreprint} or  \cite[Lemma 3.13]{Heihoff} for the details.
\end{proof}

Due to Lemma \ref{lemma3.6} and Lemma \ref{lemma3.7}, we can obtain a pointwise lower bound estimates for the solution component $v$
of \eqref{1.1} \eqref{1.7} \eqref{1.8}. This will help us obtain the bounds of $\|n(\cdot,t)\|_{L^\infty(\Omega)}$ as $t<T_{max}$.

\begin{lemma}\label{lemma3.9}Let $\int_\Omega n_0<\delta(K)$ and suppose that $T_{max}<\infty$.  Then one can find  $C>0$  such that
\begin{equation}\label{3.27}
v(x,t)\geq C\quad\quad \hbox {for all}\,\, x\in \Omega, \,t\in (0,T_{max}).
\end{equation}
\end{lemma}
 \begin{proof}[Proof]
 Let $z(x,t):=-\ln\frac{v(x,t)}{\|v_{0}\|_{L^{\infty}(\Omega)}}$. Due to \eqref{2.4} and the second equation in \eqref{1.1}, we can see that nonnegative function
 $z\in  C^{2,1}(\bar{\Omega}\times(0,T_{max}))$ satisfies
 \begin{equation}\label{3.28}
 z_t-\triangle z=-|\nabla z|^2+n-u\cdot \nabla z\leq n+\frac{|u|^2}2.
\end{equation}
Hence thanks to known regularity results on the heat equation involving sources from Orlicz spaces (\cite [Corollary 1.3]{WinkleIMRN}), the combination of \eqref{3.20} and \eqref{3.24} yields  \eqref{3.27} using a comparison principle (see the proof of \cite [Lemma 3.7]{Fuest} for details).

We note that \eqref{3.27} can also be derived by the Neumann heat semigroup smoothness estimates. Indeed, applying the variation-of-constants formula to the  equation of \eqref{3.28}, we arrive at
	\begin{equation}\label{3.29}
		\begin{split}
			z(\cdot,t)&\leq e^{t\Delta}z_{0}+\int_{0}^{t}e^{(t-s)\Delta}( n+\frac 12 |u|^2) ds\\
					\end{split}
	\end{equation}
	for all $t<T_{\max}$. Moreover, from Lemmas \ref{lemma3.7}, \ref{lemma3.6} and Lemma \ref{lemma2.1},   it follows that
	\begin{align}\label{3.30}
	&\|z(\cdot,t)\|_{L^{\infty}(\Omega)}\notag\\
\leq
&\|e^{t\Delta}z_{0}\|_{L^{\infty}(\Omega)}+\int_{0}^{t}\|e^{(t-s)\Delta}n(\cdot,s)\|_{L^{\infty}(\Omega)}ds
+
\sup_{t\in (0,T_{max})}\|u(\cdot,t)\|_{L^\infty(\Omega)} T_{\max}\\
			\leq
&\|z_{0}\|_{L^{\infty}(\Omega)}+\int_{0}^{t}\|e^{(t-s)\Delta}(n-\bar{n})(\cdot,s)
\|_{L^{\infty}(\Omega)}ds+\int_{0}^{t}\|e^{(t-s)\Delta}\bar{n}(\cdot,s)\|_{L^{\infty}(\Omega)}ds+C_1T_{\max}\notag\\
			\leq
&
\|z_{0}\|_{L^{\infty}(\Omega)}+C_2\int_{0}^{t}(1+(t-s)^{-\frac{1}{2}})e^{-\lambda_{1}(t-s)}\|(n-\bar{n})(\cdot,s)\|_{L^{2}(\Omega)}ds+
C_2T_{\max} \notag \\
\leq& C_3T_{\max}\notag		
	\end{align}
with constants $C_i>0 \, (i=1,2,3)$, readily leading to \eqref{3.27} with the desired $C$.
	\end{proof}

Thanks to the time-dependent lower bound for $v$ expressed in   Lemma \ref{lemma3.9}, the
degeneracy of the diffusion  in  the $n-$equation  of \eqref{1.1} is  excluded over finite time intervals, and thereby allows us to proceed in a quite  standard manner to derive the bounds of $\|n(\cdot,t)\|_{L^\infty(\Omega)}$.

\begin{lemma}\label{lemma3.10} Let $\int_\Omega n_0<\delta(K)$  and assume that $T_{max}<\infty$, then there exists  $C>0$ such that
\begin{equation}\label{3.31}
\|n(\cdot,t)\|_{L^\infty(\Omega)} \leq C
\end{equation}
for all $t\in (0, T_{max})$.
\end{lemma}
\begin{proof}[Proof]Let $(n,v,u,P)$ be the maximal solution obtained in Lemma \ref{lemma2.1}. Thanks to the pointwise estimate of $v$ from below and above asserted by \eqref{3.27} and  \eqref{2.4}, respectively, it follows  from  Lemma \ref{lemma2.2}  that  there exist $C_1>0, C_2>0$ such that $C_1\leq\phi(v)\leq C_2 $ and $|\phi'(v)|\leq C_2$. On the other hand,  if $T_{max}<\infty$, the boundedness of $\| \nabla v(\cdot,t)\|_{L^{\infty}(\Omega)}$ follows from \eqref{3.26}. Therefore through a Moser-type iterative argument, we can achieve \eqref{3.31}. We refer the interested reader to the proof of \cite[Proposition 1.3]{TaoJDE} for details.
\end{proof}

The sequence of lemmas in this section has now set us up to obtain our main result on global existence of classical solutions to \eqref{1.1} \eqref{1.7} \eqref{1.8} readily.

\vspace{.5cm}

\it {Proof of Theorem 1.1 (i).}\rm   \quad Assuming that $T_{max}<\infty$, it  follows from  Lemmas \ref{lemma3.10}, \ref{lemma3.8} and \ref{lemma3.7} that there exists $C>0$ such that
\begin{equation*}
		\sup\limits_{t\in (0, T_{max})}\left\{ \|n(\cdot,t)\|_{L^{\infty}(\Omega)}\!
+\!\|v(\cdot,t)\|_{W^{1,q}(\Omega)}+\|A^\beta u(\cdot,t)\|_{L^2(\Omega)}
\right\}\leq C,
	\end{equation*}
which leads to a contradiction to \eqref{2.2} and thus allows us to conclude $T_{max}=\infty$.

\section{Proof of  Theorem 1.1(ii)}

  In this section, in order to investigate the large time behavior of solutions to the  problem \eqref{1.1} \eqref{1.7} \eqref{1.8} with $\mu>0$,  we first utilize the damping effect of the logistic term  in the first equation of  \eqref{1.1} to   derive  the uniform boundedness of global classical solutions.
  To this end, similarly as the proof of Lemma \ref{lemma3.1}, we  derive the following result:
\begin{lemma}\label{lemma4.1} Let $\|v_{0}\|_{L^{\infty}(\Omega)}\leq K$.  Then for given $T>0$  and $t<\min\{T,T_{max}\}$, we have
	\begin{equation}\label{4.1}
\frac d{dt}\int_\Omega n\ln n+
\frac{\lambda(K)} 2\int_\Omega\frac{v}{n}|\nabla n|^2
\leq \frac{\Lambda^2(K)}{2\lambda(K)}
 \int_\Omega\frac nv |\nabla v|^2+\mu \int_\Omega n\ln n(1-n).
\end{equation}
\end{lemma}
To suitably control  $\int_\Omega\frac nv |\nabla v|^2 $ in the right-side of \eqref{4.1}, we require the
following result regarding the evolution of the term $\int_\Omega  \frac{|\nabla v|^2}v$:

\begin{lemma}\label{lemma4.2}For given $T>0$, there is a constant $C_1>0$ independently of $T$ such that  for all $t<\min\{T,T_{max}\}$, we have
 \begin{equation}\label{4.2}
\begin{split}
&\frac d{dt} \int_\Omega  \frac{|\nabla v|^2}v
+\frac1 {16}\int_\Omega \frac{|\nabla v|^4}{v^3}+ \int_\Omega \frac nv |\nabla v|^2\\
\leq & C_1
\int_\Omega |\nabla u|^2+ C_1\int_\Omega  {n^2}+
C_1.
\end{split}
\end{equation}
\end{lemma}
\begin{proof}[Proof]Similar as the proof of Lemma \ref{lemma3.2}, we have
\begin{equation*}
\begin{split}
&\frac d{dt} \int_\Omega  \frac{|\nabla v|^2}v
+2\int_\Omega v|D^2\ln v|^2+ \int_\Omega \frac nv |\nabla v|^2\\
\leq & -2\int_\Omega \nabla n\cdot\nabla v+  \frac1{\varepsilon}
\int_\Omega v |\nabla u|^2+\varepsilon\int_\Omega v|\Delta\ln v|^2+
2\varepsilon
\int_\Omega \frac{|\nabla v|^4}{v^3}+C_\varepsilon\int_{\Omega}v
\\
\leq & \varepsilon \int_\Omega \frac {|\Delta v|^2}v+
\frac 1{\varepsilon} \int_\Omega  {n^2}v
 +\frac1{\varepsilon}
\int_\Omega v |\nabla u|^2+2\varepsilon\int_\Omega v|D^2\ln v|^2+
2\varepsilon
\int_\Omega \frac{|\nabla v|^4}{v^3}+C_\varepsilon\int_{\Omega}v,
\end{split}
\end{equation*}
which along  with \eqref{3.16} and \eqref{3.18} yields
\begin{align}\label{4.3}
&\frac d{dt} \int_\Omega  \frac{|\nabla v|^2}v
+(\frac {1-\varepsilon} {4(7 + 4\sqrt{2})}-\varepsilon)  \int_\Omega \frac {|\Delta v|^2}v
+(\frac {1-\varepsilon} {(2 + \sqrt{2})^2}-2\varepsilon)
\int_\Omega \frac{|\nabla v|^4}{v^3}+ \int_\Omega \frac nv |\nabla v|^2 \notag\\
\leq & \frac1{\varepsilon}
\int_\Omega v |\nabla u|^2+  \frac 1{\varepsilon} \int_\Omega  {n^2}v+
C_\varepsilon\int_{\Omega}v\\
\leq & \frac K {\varepsilon}
\int_\Omega |\nabla u|^2+  \frac K {\varepsilon} \int_\Omega  {n^2}+
C_\varepsilon\int_{\Omega}v_0. \notag
\end{align}
Hence, choosing $\varepsilon$ sufficiently small in \eqref{4.3}, one can see that the inequality \eqref{4.2}  holds with $C_1:=
\frac K {\varepsilon}+ C_\varepsilon\int_{\Omega}v_0$.
\end{proof}

We note that the derivation of \eqref{3.14} actually does not require any smallness assumption of $n$, but uses Lemma \ref{lemma2.3} instead.
Hence, drawing on  \eqref{4.1}, \eqref{4.2} and \eqref{3.14}, we can obtain the bounds of  $ \int^{t}_{(t-\tau)_+} \int_\Omega  \frac{|\nabla v|^4}{v^3}$.

\begin{lemma}\label{lemma4.3} For any given $T>0$ and $0<\tau\leq 1$, there exists a constant $C>0$  such that
\begin{equation}\label{4.4}
\begin{split}
\int_\Omega n(\cdot,t)|\ln n(\cdot,t)|+\int_\Omega  \frac{|\nabla v(\cdot,t)|^2}{v(\cdot,t)}+\int_\Omega |u(\cdot,t)|^2\leq C
\quad \hbox{for all}\,  \,t\in (0,\min\{T,T_{max}\})
\end{split}
\end{equation}
as well as
\begin{equation}\label{4.5}
\begin{split}
 \int^{t}_{(t-\tau)_+} \int_\Omega  \frac{|\nabla v|^4}{v^3}+\int^{t}_{(t-\tau)_+} \int_\Omega |\nabla u|^2 \leq C
\quad \hbox{for all}\,  \,t\in (0,\min\{T,T_{max}\}).
\end{split}
\end{equation}
\end{lemma}

\begin{proof}[Proof]
 Combining \eqref{4.1} and \eqref{4.2}, we obtain
  \begin{equation}\label{4.6}
\begin{split}
&\frac d{dt} (\int_\Omega n\ln n+
\frac{\Lambda^2(K)}{2\lambda(K)}\int_\Omega  \frac{|\nabla v|^2}v)+
\frac{\lambda(K)} 2\int_\Omega\frac{v}{n}|\nabla n|^2+
\frac{\Lambda^2(K)}{32\lambda(K)}
\int_\Omega \frac{|\nabla v|^4}{v^3}
\\
\leq &
\mu \int_\Omega n(1-n)\ln n+ \frac{C_1\Lambda^2(K)}{2\lambda(K)}\{\int_\Omega n^2+1+
\int_\Omega  |\nabla u|^2\}
\end{split}
\end{equation}
for all  $t<\min\{T,T_{max}\}$.

 This further combines with \eqref{3.14} to yield
\begin{align}\label{4.7}
&\frac d{dt} \int_\Omega (n\ln n+
\frac{\Lambda^2(K)}{2\lambda(K)} \frac{|\nabla v|^2}v
+\frac{\Lambda^2(K)C_1|u|^2}{\lambda(K)}
)+\int_\Omega\frac{C_1\Lambda^2(K)|\nabla u|^2}{2\lambda(K)}+
\frac{\Lambda^2(K)}{32\lambda(K)}
 \frac{|\nabla v|^4}{v^3} + n\ln n \notag
\\
\leq &
\mu \int_\Omega n(1-n)\ln n+ \frac{C_1\Lambda^2(K)}{\lambda(K)}\left\{\int_\Omega n^2+1+
\int_\Omega n \cdot\int_\Omega n\ln\frac n{\overline n}+(\int_\Omega n)^2\right\}+\int_\Omega n\ln n \notag\\
=&
-\mu \int_\Omega n^2\ln n+(\mu+1) \int_\Omega n\ln n+
 \frac{C_1\Lambda^2(K)}{\lambda(K)}\left\{\int_\Omega n^2+1+
\int_\Omega n \cdot\int_\Omega n\ln n\right. \notag\\
&\left.
+(\int_\Omega n)^2(1+\ln |\Omega|)- (\int_\Omega n)^2\ln(\int_\Omega n )\right\}\\
\leq &C_2 \notag
\end{align}
 with some $C_2>0$, where we have used the facts that for all $\varepsilon>0$,  $n^2\leq \varepsilon n^2\ln n +C(\varepsilon)$,
 $n\ln n\leq \varepsilon n^2\ln n +C(\varepsilon)$ and $\displaystyle\min_{0<x<1}x^2\ln x=-(2e)^{-1}$.

Furthermore, by the Young inequality and  \eqref{3.16}, one can conclude that
there exist constants $C_3>0$ and $C_4>0$ such that for
$$\mathcal{F}_2(t):=   \int_\Omega (n\ln n+
\frac{\Lambda^2(K)}{2\lambda(K)} \frac{|\nabla v|^2}v
+\frac{\Lambda^2(K)C_1}{\lambda(K)}|u|^2
),
$$
we have
$$
\mathcal{F}_2'(t)+C_3 \mathcal{F}_2(t)+ \frac{\Lambda^2(K)}{64\lambda(K)}
 \int_\Omega \frac{|\nabla v|^4}{v^3}\leq C_4.
$$
Using an ODE comparison argument now gives us a constant $C_5>0$ such that
$$
\int_\Omega n(\cdot,t)|\ln n(\cdot,t)|+\int_\Omega  \frac{|\nabla v(\cdot,t)|^2}{v(\cdot,t)}+\int_\Omega |u(\cdot,t)|^2\leq C_5
\quad \hbox{for all}\,  \,t<\min\{T,T_{max}\}
$$
as well as
$$
 \int^{t}_{(t-\tau)_+} \int_\Omega  \frac{|\nabla v(\cdot,s)|^4}{v^3(\cdot,s)} \leq C_5
\quad \hbox{for all}\,  \,t\in (0,\min\{T,T_{max}\}).
$$
\end{proof}
Now drawing on the spatio-temporal estimates of $\frac{|\nabla v|^4}{v^3}$ in Lemma \ref{lemma4.3},
we take advantage of the logistic term in the first equation of \eqref{1.1} to derive a priori estimates of $n$ in the $L^p(\Omega)$-norm in a manner quite similar to
that of Lemma \ref{lemma3.6}, using Lemma \ref{lemma2.5}.

\begin{lemma}\label{lemma4.4} For any $T>0$ and $p \geq  2$, there exists  $C(p)>0$ such that
\begin{equation}\label{4.8}
\|n(\cdot,t)\|_{L^p(\Omega)} \leq C(p)
\end{equation}
for all $t\in (0, \min\{T_{max},T\})$.
\end{lemma}
\begin{proof}Thanks to the Young inequality, Lemma \ref{lemma2.4} and $\nabla\cdot u=0$,
we use the first equation in \eqref{1.1} and integrate by parts  to see that  there exists $C_1=p^2\Gamma^p$
with a constant $\Gamma>1$ such that
\begin{align}\label{4.9}
&\frac d{dt}\int_\Omega n^p\notag\\
=& p\int_\Omega n^{p-1}\{ \triangle(n\phi(v))-u\cdot \nabla n+\mu n(1-n)\}\\
=& -p(p-1) \int_\Omega (\phi(v)n^{p-2}|\nabla n|^2
+ n^{p-1} \phi'(v)\nabla n\cdot \nabla v )+ \mu p \int_\Omega n^p- \mu p \int_\Omega n^{p+1}  \notag\\
\leq & -\frac {p(p-1)\lambda(K)}2 \int_\Omega vn^{p-2}|\nabla n|^2+\frac {p(p-1)\Lambda^2(K)}{2\lambda(K)}
 \int_\Omega \frac{n^p}v |\nabla v|^2+ \mu p \int_\Omega n^p- \mu p \int_\Omega n^{p+1}    \notag\\
 \leq &C_1
 \int_\Omega n v+  C_1 \left\{\int_\Omega n^p+\left(\int_\Omega  n\right)^{2p-1} \right\}
 \cdot  \int_\Omega \frac{|\nabla v|^4}{v^3}+ \mu p \int_\Omega n^p- \mu p \int_\Omega n^{p+1}   \notag\\
 \leq &C_1 \int_\Omega n v+ 2^p\mu p|\Omega|+ C_1 \left(\int_\Omega  n\right)^{2p-1}  \cdot  \int_\Omega \frac{|\nabla v|^4}{v^3}+
  C_1 \int_\Omega \frac{|\nabla v|^4}{v^3}\cdot\int_\Omega n^p-\frac{\mu p}  {2|\Omega|^{\frac1 p}}\cdot\left\{\int_\Omega n^p\right\}^{\frac{p+1}p}   \notag
\end{align}
for all $t\in (0, \min\{T_{max},T\})$.

Now let $y(t):=\int_\Omega n^p(\cdot,t)$ and
$$
g(t):=C_1 \int_\Omega n v+ 2^p\mu p|\Omega|+ C_1 \left(\int_\Omega  n\right)^{2p-1}  \cdot  \int_\Omega \frac{|\nabla v|^4}{v^3},
\quad  h(t):=C_1 \int_\Omega \frac{|\nabla v|^4}{v^3}.
$$
Then we have
\begin{equation}\label{4.10}
y'(t)+ \frac{\mu p}  {2|\Omega|^{\frac1 p}} y^{\frac{p+1}p}(t)\leq g(t)+h(t)y(t)
\end{equation}
for all $t\in (0, \min\{T_{max},T\})$.

In view of \eqref{4.5}, one can pick
$C_2(\tau)>0$ such that
$\frac1 {\tau}\int^t_{(t-\tau)_+}h(s)ds \leq C_2(\tau)$ for each $\tau\leq 1$. Hence,
  thanks to $\frac{p+1}p>1$,  an application of Young's inequality implies that
  $$
 \frac{\mu p}  {2|\Omega|^{\frac1 p}} y^{\frac{p+1}p}(t)\geq 2C_2(\tau) y(t)-C_3(\tau)
  $$
 with some $C_3(\tau)>0$, and thereby \eqref{4.10} becomes
 \begin{equation}\label{4.11}
y'(t)+  2C_2(\tau) y(t)\leq h(t)y(t) +g(t)+C_3(\tau).
\end{equation}

 Since $\int_0^\infty \int_\Omega n v\leq  \int_\Omega v_0$,  \eqref{2.4} and
$ \int^t_{(t-\tau)_+}\int_\Omega \frac{|\nabla v|^4}{v^3}\leq C_4$ by \eqref{4.5}, we can see that
 $$\int^t_{(t-\tau)_+}h(s)ds\leq  \tau C_2(\tau)$$ as well as
 \begin{equation}\label{4.12}
 \begin{split}
 \int^t_{(t-\tau)_+}g(s)ds\leq & C_1 \int_\Omega v_0+ 2^p\mu p|\Omega|\tau+ C_1 \left(\int_\Omega  n\right)^{2p-1}  \cdot  \int^t_{(t-\tau)_+}\int_\Omega \frac{|\nabla v|^4}{v^3}\\
\leq & C_5:=C_1 \int_\Omega v_0+ 2^p\mu p|\Omega|\tau+C_1 C_4 \left(
\max\{\int_\Omega  n_0,|\Omega|\} \right)^{2p-1},
\end{split}
\end{equation}
  and an application of Lemma \ref{lemma2.5} to $b_1= \tau C_2(\tau)$, $\varrho=\tau C_2(\tau)$ and
  $c_1=\tau C_3(\tau)+ C_5$  readily establishes \eqref{4.8} upon an appropriate
choice of the constant $C(p)$.
\end{proof}

At this point we are ready to ascertain the global existence of classical solutions to \eqref{1.1} \eqref{1.7} \eqref{1.8} in the case $\mu>0$.

\begin{lemma}\label{lemma4.5}
The problem \eqref{1.1} \eqref{1.7} \eqref{1.8} with  $\mu>0$ admits a global classical solution.
\end{lemma}
\begin{proof}
We apply a reasoning analogous to that in the proof of part (i) of Theorem 1.1. Assuming that
 $T_{max}<\infty$, and applying Lemma \ref{lemma4.4} with  $T=T_{max}$, we have
\begin{equation}\label{4.13}
\|n(\cdot,t)\|_{L^p(\Omega)} \leq C(p)
\end{equation}
with some $C(p)>0$
for all $t\in (0, T_{max})$. Furthermore, it follows from Lemmas \ref{lemma3.10}, \ref{lemma3.8} and \ref{lemma3.7} that there exists $C>0$ such that
\begin{equation*}
		\sup\limits_{t\in (0, T_{max})}\left\{ \|n(\cdot,t)\|_{L^{\infty}(\Omega)}\!
+\!\|v(\cdot,t)\|_{W^{1,q}(\Omega)}+\|A^\beta u(\cdot,t)\|_{L^2(\Omega)}
\right\}\leq C,
	\end{equation*}
which leads to a contradiction to \eqref{2.2} and thus the conclusion that $T_{max}=\infty$. Moreover, from the proof of Lemmas \ref{lemma3.8} and \ref{lemma3.7}, one can conclude that  this global solution is bounded in the sense that
 \begin{equation*}
	\sup_{t>0}\left\{\!\|v(\cdot,t)\|_{W^{1,\infty}(\Omega)}+\|A^\beta u(\cdot,t)\|_{L^2(\Omega)}\right\}< \infty,
\end{equation*}
and thus complete the proof.
\end{proof}

Towards the derivation of the large time behavior  of the solution,  
  we shall  consider the evolution of the quantity
 $\int_\Omega \frac{|\nabla v(\cdot,t)|^4}{v^3(\cdot,t)}$.

\begin{lemma}\label{lemma4.6}
	Let $(n,v,u,P)$ be the global classical solution of the problem \eqref{1.1} \eqref{1.7} \eqref{1.8} with $\mu>0$.
Then there exists $C>0$ such that
\begin{equation}\label{4.14}
\int_\Omega \frac{|\nabla v(\cdot,t)|^4}{v^3(\cdot,t)}\le C \quad\quad \hbox{for all}\,\, t>0.
\end{equation}
\end{lemma}
\begin{proof}
By Lemmas \ref{lemma4.4} and \ref{lemma3.7}, we have
 \begin{equation}\label{4.15}
\|A^\beta u(\cdot,t)\|_{L^2(\Omega)} \leq C_1
\end{equation}
for some $C_1>0$.

As in the proof of Lemma 2.3  in \cite{WinklerJDE2024} and by Lemma 3.5 in \cite{WinklerDCDSB(2022)},
we find $C_i>0$, $(i=2,\ldots,5)$, such that
\begin{equation*}
\begin{split}
&\frac d{dt} \int_\Omega  \frac{|\nabla v|^4}{v^3}+C_2\int_\Omega v^{-3}|\nabla v|^2 |D^2 v|^2 + C_2
\int_\Omega \frac{|\nabla v|^6}{ v^5}
\\
\leq & -4
\int_\Omega v^{-3}|\nabla v|^2 \nabla v\cdot \nabla ( u\cdot \nabla v )
+3
\int_\Omega v^{-4}|\nabla v|^4  (u\cdot \nabla v )+
  2\int_{\partial\Omega}v^{-3}|\nabla v|^2
\frac {\partial |\nabla v|^2}{\partial \nu}+C_3\int_\Omega n^3 v \\
\leq &
-4\int_\Omega v^{-3}|\nabla v|^2[ ( \nabla v \otimes u):D^2 v+(\nabla v\otimes \nabla v):\nabla u]
+3\int_\Omega v^{-4}|\nabla v|^4  (u\cdot \nabla v )\\
& +\frac {C_2}2\int_\Omega v^{-3}|\nabla v|^2 |D^2 v|^2
+\frac {C_2}2
\int_\Omega \frac{|\nabla v|^6}{ v^5}+C_4\int_\Omega (n^3+1)v \\
\leq &
4C_1\int_\Omega
\left(
v^{-3}|\nabla v|^3|D^2 v|+
v^{-3}|\nabla v|^4+ v^{-4}|\nabla v|^5\right) +\frac {C_2}2\int_\Omega v^{-3}|\nabla v|^2 |D^2 v|^2
\\
&+\frac {C_2}2
\int_\Omega \frac{|\nabla v|^6}{ v^5}+C_4\int_\Omega (n^3+1)v \\
\leq &
\frac {3C_2}4\int_\Omega v^{-3}|\nabla v|^2 |D^2 v|^2
+\frac {3C_2}4
\int_\Omega \frac{|\nabla v|^6}{ v^5}+C_5\int_\Omega (n^3+1)v,
\end{split}
\end{equation*}
and thus
\begin{equation}\label{4.16}
\frac d{dt} \int_\Omega  \frac{|\nabla v|^4}{v^3}+\frac {3C_2}4
\int_\Omega \frac{|\nabla v|^6}{ v^5}
\leq C_5\int_\Omega (n^3+1)v.
\end{equation}
Furthermore, by the inequality
$$
 \int_\Omega\frac{|\nabla v|^4}{ v^3}\leq \frac {3C_2}4
\int_\Omega \frac{|\nabla v|^6}{ v^5}+C_6 \int_\Omega   v,
$$
we get
\begin{equation}\label{4.17}
\frac d{dt} \int_\Omega  \frac{|\nabla v|^4}{v^3}+
\int_\Omega \frac{|\nabla v|^4}{ v^3}
\leq (C_5+C_6)\int_\Omega (n^3+1)v.
\end{equation}
 Now, for any given $t\geq 1$, it follows from \eqref{4.5} that
\begin{equation*}
\begin{split}
 \int^{t}_{t-1} \int_\Omega  \frac{|\nabla v|^4}{v^3} \leq C_7,
\end{split}
\end{equation*}
and thus there exists $t_0\in(t-1,t)$ such that $\int_\Omega  \frac{|\nabla v(\cdot,t_0)|^4}{v^3(\cdot,t_0)} \leq C_7$.
Hence from \eqref{4.17}, we have
\begin{equation}\label{4.18}
\int_\Omega\frac{|\nabla v(\cdot,t)|^4}{ v^3(\cdot,t)}\leq
\int_\Omega\frac{|\nabla v(\cdot,t_0)|^4}{ v^3(\cdot,t_0)}  +   (C_5+C_6)\int^{t}_{t_0}  \int_\Omega (n^3+1)v.
\end{equation}
Finally, with the help of \eqref{4.8} and \eqref{2.4}, we arrive at \eqref{4.14}. On the other hand, it follows from
 \eqref{4.17} that for all $t< 1$,
\begin{equation*}
\int_\Omega\frac{|\nabla v(\cdot,t)|^4}{ v^3(\cdot,t)}\leq
\int_\Omega\frac{|\nabla v_0(\cdot)|^4}{ v_0^3(\cdot)}  +   (C_5+C_6)\int^{1}_{0}  \int_\Omega (n^3+1)v,
\end{equation*}
and  \eqref{4.14} is hence established readily.
\end{proof}

Now we shall establish the lower bound of
the time average value of $\|n(\cdot,t)\|_{L^1(\Omega)}$ at large time scale with the help of the boundedness of $\int_\Omega \frac{|\nabla v(\cdot,t)|^4}{v^3(\cdot,t)}$ at least when $\|v_0\|_{L^1(\Omega)}$ is sufficiently small.
\begin{lemma}\label{lemma4.7}
Let $(n,v,u,P)$ be the global classical solution of \eqref{1.1} \eqref{1.7} \eqref{1.8} with $\mu>0$.
Then for any given $K>0$, one can pick $\delta=\delta(\mu,K)>0$  such that whenever  $\|v_{0}\|_{L^{\infty}(\Omega)}\leq K$ and $\|v_0\|_{L^1(\Omega)}<\delta$, we have
\begin{equation}\label{4.19}
\displaystyle\lim _{ \overline{T\rightarrow \infty}}\frac1T\int^T_1\|n(\cdot,t)\|_{L^1(\Omega)}\geq \frac { |\Omega|} 3.
\end{equation}
\end{lemma}
\begin{proof}[Proof] Multiplying the first equation in \eqref{1.1} by $n^{-1}$, an application of Lemma \ref{lemma2.2}, the Young inequality and $\nabla\cdot u=0$ shows that
\begin{equation}\label{4.20}
\begin{split}
\frac d{dt}\int_\Omega \ln n
&=\int_\Omega \frac 1 n
(\Delta (n\phi(v))-u\cdot\nabla n)+ \mu|\Omega|- \mu\int_\Omega  n\\
&=\int_\Omega\frac{\nabla n}{n^2}
(\phi(v) \nabla n+\phi'(v)n \nabla v )+ \mu|\Omega|- \mu\int_\Omega  n\\
&\geq \frac12\int_\Omega \frac{\phi(v)} {n^2}|\nabla n|^2-\frac12
\int_\Omega \frac{|\phi'(v)|^2} {\phi(v)} |\nabla v|^2+ \mu|\Omega|- \mu\int_\Omega  n\\
&\geq-
\frac{\Lambda^2(K)}{2\lambda(K)}\int_\Omega \frac {|\nabla v|^2}v + \mu|\Omega|- \mu\int_\Omega  n.
\end{split}
\end{equation}
On the other hand, revisiting the proof of  \eqref{4.14}, one can conclude that there exists $C_1>0$,  which may be dependent on  $K$ but independent of $\|v_0\|_{L^1(\Omega)}$,  such that
$$\int_\Omega \frac{|\nabla v(\cdot,t)|^4}{v^3(\cdot,t)}\le C_1  \quad\quad \hbox{for all}\,\, t>0.
$$
Therefore by the H\"{o}lder inequality,  we get
\begin{equation}\label{4.21}
\begin{split}
\frac{\Lambda^2(K)}{2\lambda(K)}\int_\Omega \frac {|\nabla v|^2}v
&\leq
  \frac{\Lambda^2(K)}{2\lambda(K)}\left\{ \int_\Omega\frac {|\nabla v|^4}{v^3}\right\}^{\frac 12}
\left\{\int_\Omega  v \right\} ^{\frac 12}\\
&\leq  \frac{\Lambda^2(K)}{2\lambda(K)}C_1^{\frac 12} \left\{\int_\Omega  v_0 \right\} ^{\frac 12}.
\end{split}
\end{equation}
Pick $\delta>0$ sufficiently small such that
$$ \frac{\Lambda^2(K)}{2\lambda(K)}C_1^{\frac 12} \delta ^{\frac 12}\leq \frac{ \mu|\Omega|}2.
$$
 Then from \eqref{4.20} and \eqref{4.21} it follows that
\begin{equation}\label{4.22}
\begin{split}
\frac d{dt}\int_\Omega \ln n
\geq
\frac{ \mu|\Omega|}2- \mu\int_\Omega  n.
\end{split}
\end{equation}
For any $T>2$, integrating \eqref{4.22} over $(1,T)$, we readily obtain that
\begin{equation}\label{4.23}
\begin{split}
\frac 1T \int^{T}_1\int_\Omega n
&\geq  \frac { |\Omega|} 2-\frac 1{T\mu}\int_\Omega \ln n(\cdot,T)  +\frac 1{T\mu} \int_\Omega \ln n(\cdot,1)\\
& \geq
\frac  { |\Omega|} 2-\frac 1{T\mu}\int_\Omega n (\cdot,T) +\frac 1{T\mu} \int_\Omega \ln n(\cdot,1)\\
& \geq
\frac  { |\Omega|} 2-\frac 1{T\mu}\max\{\int_\Omega  n_0,|\Omega|\}  +\frac 1{T\mu} \int_\Omega \ln n(\cdot,1),
\end{split}
\end{equation}
where  \eqref{2.3b} and $\ln x<x$ for $x>0$ are used, and therefore get \eqref{4.19} immediately.
\end{proof}

Noting that $\|v(\cdot,t)\|_{L^1(\Omega)}$  decreases with respect to time $t>0$,
the lower bound of the time average value of $\|n(\cdot,t)\|_{L^1(\Omega)}$ expressed in \eqref{4.19} can inform us of
the decay properties of $\|v(\cdot,t)\|_{W^{1,\infty}(\Omega)}$.
\begin{lemma}\label{lemma4.8}
 Let $K>0$ and $\|v_0\|_{L^1(\Omega)}<\delta$ be given as in Lemma \ref{lemma4.7}. Then we have
 \begin{equation}\label{4.24}
\displaystyle\lim_{t\rightarrow \infty} \|v(\cdot,t)\|_{W^{1,\infty}(\Omega)} =0.
\end{equation}
\end{lemma}
\begin{proof}[Proof]
Multiplying the second equation in \eqref{1.1} by $v$, we get
 \begin{equation*}
 \frac d{dt} \int_\Omega v^2 + \int_\Omega |\nabla v|^2 =-2\int_\Omega nv^2
\end{equation*}
and then
\begin{equation}\label{4.25}
\int^T_0 \int_\Omega |\nabla v|^2+2\int^T_0 \int_\Omega nv^2\leq \int_\Omega v_0^2
\end{equation}
for all $T>0$.
 Moreover,  by the Poincar\'{e} inequality and  H\"{o}lder's inequality, we have
 \begin{equation}\label{4.26}
\begin{split}
|\Omega|\int^T_1 \int_\Omega\overline{n}\cdot   \overline{v}=& \int^T_0 \int_\Omega nv- \int^T_0 \int_\Omega n(v- \overline{v})\\
  \leq & \int_\Omega v_0+  \int^T_0  \|n\|_{L^2(\Omega)} \|v-\overline{v}\|_{L^2(\Omega)}\\
  \leq & \int_\Omega v_0+ C_p\sup_{t>0}\|n(\cdot,t)\|_{L^2(\Omega)} \int^T_0  \|\nabla v\|_{L^2(\Omega)}\\
  \leq & \int_\Omega v_0+ C_p\sup_{t>0}\|n(\cdot,t)\|_{L^2(\Omega)} \left\{\int^T_0  \|\nabla v\|^2_{L^2(\Omega)}\right\}^{\frac12}T^{\frac12}\\
  \leq & \int_\Omega v_0+ C_p\sup_{t>0}\|n(\cdot,t)\|_{L^2(\Omega)}  \| v_0\|_{L^2(\Omega)}T^{\frac12},\\
 \end{split}
\end{equation}
which, together with the decreasing behavior of $\|v(\cdot,t)\|_{L^1(\Omega)}$, leads us to
 \begin{equation}\label{4.27}
\begin{split}
\int_\Omega v(\cdot,T)\cdot \frac1T \int^T_1 \int_\Omega n   \leq \frac{|\Omega|}T \int_\Omega v_0+  C_p |\Omega| \sup_{t>0}\|n(\cdot,t)\|_{L^2(\Omega)}  \| v_0\|_{L^2(\Omega)}T^{-\frac12}.
 \end{split}
\end{equation}
 Therefore combining \eqref{4.27} with \eqref{4.23}, we obtain that
  \begin{equation}\label{4.28}
  \begin{split}
&\left\{ \frac {|\Omega|}2-\frac 1{T\mu}\max\{\int_\Omega  n_0,|\Omega|\}  +\frac 1{T\mu} \int_\Omega \ln n_0 \right\}\int_\Omega v(\cdot,T)\\
  \leq &\frac{|\Omega|}T \int_\Omega v_0+ C_p |\Omega| \sup_{t>0}\|n(\cdot,t)\|_{L^2(\Omega)}  \| v_0\|_{L^2(\Omega)}T^{-\frac12}.
 \end{split}
\end{equation}

Now since $ \displaystyle\lim_{T\rightarrow \infty}\{\frac 1{T\mu}\max\{\int_\Omega  n_0,|\Omega|\}-\frac 1{T\mu} \int_\Omega \ln n(\cdot,1)  \}=0$, one can find $T_1>2$ such that for all $T>T_1$, we have
\begin{equation*}
  \begin{split}
 \frac {|\Omega|}4 \int_\Omega v(\cdot,T)
  \leq \frac{|\Omega|}T \int_\Omega v_0+ C_p |\Omega| \sup_{t>0}\|n(\cdot,t)\|_{L^2(\Omega)}  \| v_0\|_{L^2(\Omega)}T^{-\frac12}
 \end{split}
\end{equation*}
and thus
 \begin{equation}\label{4.29}
\displaystyle\lim_{t\rightarrow \infty} \|v(\cdot,t)\|_{L^{1}(\Omega)} =0,
\end{equation}
 which can ensure that   \eqref{4.24} holds by the interpolation inequality and  the smoothing properties of the Neumann heat semigroup. Indeed,
  thanks to the boundedness of  $\|v(\cdot,t)\|_{W^{1,\infty}(\Omega)} $ asserted in Lemma \ref{lemma4.5},  we have
  \begin{equation}\label{4.29a}
\displaystyle\lim_{t\rightarrow \infty} \|v(\cdot,t)\|_{L^{\infty}(\Omega)} =0.
\end{equation}
Since
$$\int_\Omega |\nabla v(\cdot,t)|^2=\int_\Omega \frac{|\nabla v(\cdot,t)|^2}{v(\cdot,t)}\cdot
 \|v(\cdot,t)\|_{L^{\infty}(\Omega)}
 \quad\quad \hbox{for all}\,\, t>0,
 $$
 the combination of \eqref{4.29a} and  \eqref{4.4} leads to
 $
\displaystyle\lim_{t\rightarrow \infty} \| \nabla v(\cdot,t)\|_{L^2(\Omega)} =0,
$
and thus for all $p>2$
 \begin{equation}\label{4.29b}
\displaystyle\lim_{t\rightarrow \infty} \| \nabla v(\cdot,t)\|_{L^p(\Omega)} =0
\end{equation}
 by the H\"{o}lder inequality and the boundedness of  $\|v(\cdot,t)\|_{W^{1,\infty}(\Omega)}$. Moreover by
  employing the smoothing properties of the Neumann heat semigroup (\cite[Lemma 2.1(iii)]{Cao}), we can obtain that
  \begin{align}\label{4.29c}
&\|\nabla v(\cdot,t)\|_{L^\infty(\Omega)}\notag\\
\leq &  C_1(1+(t-t_1)^{-\frac14}) e^{-\lambda_1(t-t_1)}
\|\nabla v(\cdot,t_1)\|_{L^4(\Omega)}
+\int_{t_1}^t
\|\nabla e^{(t-s)\Delta}\{n(\cdot,s)v(\cdot,s)\}\|_{L^\infty(\Omega)}\notag\\
&\quad+\int_{t_1}^t\|\nabla e^{(t-s)\Delta}
(u\cdot\nabla v)(\cdot,s)\|_{L^\infty(\Omega)}\\
\leq &
 C_1(1+(t-t_1)^{-\frac14}) e^{-\lambda_1(t-t_1)}
\|\nabla v(\cdot,t_1)\|_{L^4(\Omega)}
+C_1\int_{t_1}^t(1+(t-s)^{-\frac 34})
e^{-\lambda_1(t-s)}\|n(\cdot,s)v(\cdot,s)\|_{L^4(\Omega)}\notag\\
&\quad+C_1\int_{t_1}^t(1+(t-s)^{-\frac3 4})
e^{-\lambda_1(t-s)}\|u(\cdot,s)\|_{L^\infty(\Omega)}
\|\nabla v(\cdot,s)\|_{L^4(\Omega)}\notag\\
\leq &
C_1(1+(t-t_1)^{-\frac14}) e^{-\lambda_1(t-t_1)}
\|\nabla v(\cdot,t_1)\|_{L^4(\Omega)}
+C_2 \sup\limits_{t\geq t_1}( \|v(\cdot,t)\|_{L^8(\Omega)}+\|\nabla v(\cdot,t)\|_{L^4(\Omega)})\notag
\end{align}
  with  $C_2:=C_1\int_0^\infty(1+\sigma^{-\frac34})
e^{-\lambda_1\sigma}d\sigma\cdot \sup\limits_{t\geq t_1} (\|u(\cdot,t)\|_{L^\infty(\Omega))}+ \|n(\cdot,t)\|_{L^8(\Omega)})$, the finite of which is warranted by  Lemma \ref{lemma4.5} and  Lemma \ref{lemma4.4}.
Hence due to \eqref{4.29b} and \eqref{4.29a}, we have
 \begin{equation}\label{4.29d}
\displaystyle\lim_{t\rightarrow \infty} \| \nabla v(\cdot,t)\|_{L^\infty(\Omega)} =0,
\end{equation}
and thereby arrive at  \eqref{4.24}.
\end{proof}

At this juncture, the decay feature of $\int_\Omega \frac {|\nabla v(\cdot,t)|^2}{v(\cdot,t)} $ enables us to achieve a comprehensive understanding of the large-time behavior of the solution component $n$.

\begin{lemma}\label{lemma4.9}
Let the assumption of Lemma 4.7 hold. Then  we have
 \begin{equation}\label{4.30}
\displaystyle\lim_{ \overline{t\rightarrow \infty}} \|n(\cdot,t)-1\|_{L^2(\Omega)}= 0.
\end{equation}
\end{lemma}

\begin{proof}[Proof] By the H\"{o}lder inequality, we have
\begin{equation*}
\begin{split}
\int_\Omega \frac {|\nabla v|^2}v
&\leq
  \left\{ \int_\Omega\frac {|\nabla v|^4}{v^3}\right\}^{\frac 12}
\left\{\int_\Omega  v \right\} ^{\frac 12}.
\end{split}
\end{equation*}
 It follows from \eqref{4.14} and \eqref{4.24} that
\begin{equation}\label{4.31}
\begin{split}
\displaystyle\lim_{t\rightarrow \infty}
\int_\Omega \frac {|\nabla v(\cdot,t)|^2}{v(\cdot,t)} =0.
\end{split}
\end{equation}

On the other hand, testing the first equation in \eqref{1.1} by $1-n^{-1}$, we get
\begin{equation*}
\begin{split}
\frac d{dt}\int_\Omega(n-1-\ln n)
&=\int_\Omega (1-\frac 1 n)
(\Delta (n\phi(v))-u\cdot\nabla n)- \mu\int_\Omega(1- n)^2\\
&=-\int_\Omega\frac{\nabla n}{n^2}
(\phi(v) \nabla n+\phi'(v)n \nabla v )- \mu\int_\Omega(1- n)^2\\
&\leq -\frac12\int_\Omega \frac{\phi(v)} {n^2}|\nabla n|^2+ \frac12
\int_\Omega \frac{|\phi'(v)|^2} {\phi(v)} |\nabla v|^2- \mu\int_\Omega(1- n)^2\\
&\leq
\frac{\Lambda^2(K)}{2\lambda(K)}\int_\Omega \frac {|\nabla v|^2}v - \mu\int_\Omega(1- n)^2,
\end{split}
\end{equation*}
by Lemma \ref{lemma2.2}, the Young inequality and $\nabla\cdot u=0$,
and hence
\begin{equation}\label{4.32}
\begin{split}
\frac d{dt}\int_\Omega(n-1-\ln n)
+ \mu\int_\Omega(n-1)^2\leq
\frac{\Lambda^2(K)}{2\lambda(K)}\int_\Omega \frac {|\nabla v|^2}v.
\end{split}
\end{equation}
Thanks to the nonnegativity of the function $x-1-\ln x $ for all $x>0$,  integrating  \eqref{4.32} over $(t_1,t_2)$
leads to
\begin{equation}\label{4.33}
\begin{split}
&\frac{1}{t_2-t_1}\int^{t_2}_{t_1}\int_\Omega(n-1)^2\\
\leq &\frac{1}{\mu(t_2-t_1)} \int_\Omega(n(\cdot,t_1)-1-\ln n(\cdot,t_1))
+
\frac{1}{t_2-t_1}\frac{\Lambda^2(K)}{\mu\lambda(K)}\int^{t_2}_{t_1} \int_\Omega \frac {|\nabla v|^2}v
\\
\leq& \frac{1}{\mu(t_2-t_1)} \int_\Omega(n(\cdot,t_1)-1-\ln n(\cdot,t_1))
+
\frac{\Lambda^2(K)}{\mu\lambda(K)}\sup_{s\in (t_1,t_2)} \int_\Omega \frac {|\nabla v(\cdot,s)|^2}{v(\cdot,s)}\\
\end{split}
\end{equation}
for all $0<t_1<t_2$, Now, for any given $\varepsilon>0$, by \eqref{4.31}, there exists $t_3>0$ such that for all $t>t_3$,
$$\frac{\Lambda^2(K)}{\mu\lambda(K)}\sup_{t\geq t_3} \int_\Omega \frac {|\nabla v(\cdot,t)|^2}{v(\cdot,t)}< \frac \varepsilon 2.
$$
Thereafter, one can pick $t_4>t_3+1$ sufficiently large such that
$$\frac{1}{\mu(t_4-t_3)} \int_\Omega(n(\cdot,t_3)-1-\ln n(\cdot,t_3))< \frac \varepsilon 2.
$$
Hence from \eqref{4.33}, we have
\begin{equation*}
\begin{split}
\frac{1}{t_4-t_3}\int^{t_4}_{t_3}\int_\Omega(n-1)^2
< \varepsilon\end{split}
\end{equation*}
and then get $\tilde{t}_3\in (t_3,t_4)$ such that
\begin{equation}
\int_\Omega(n(\cdot,\tilde{t}_3)-1)^2
< \varepsilon.
\end{equation}
Proceeding as above, we can find a sequence $\{\tilde{t}_k\}_{k\geq 3}$ fulfilling
$\tilde{t}_k\rightarrow \infty$ as $k\rightarrow \infty$ and
\begin{equation*}
\int_\Omega(n(\cdot,\tilde{t}_k)-1)^2< \varepsilon,
\end{equation*}
which implies \eqref{4.30} and completes the proof.
\end{proof}

\begin{lemma}\label{lemma4.10}
Let the assumption of Lemma 4.7 hold. Then  we have
 \begin{equation}\label{4.35}
\displaystyle\lim_{ \overline{t\rightarrow \infty}} \|u(\cdot,t)\|_{W^{1,2}(\Omega)}= 0.
\end{equation}
\end{lemma}
\begin{proof}[Proof]
Testing the third equation in \eqref{1.1} against $u$, by $\nabla\cdot u=0$ and
$\int_\Omega u\cdot((u\cdot\nabla)u)=0$, we have
\begin{equation*}
\begin{split}
\frac1 2\frac d{dt}\int_\Omega|u|^2+\int_\Omega|\nabla u|^2&=
\int_\Omega (n-1)u\cdot\nabla\Phi\\
&\leq
\|\nabla\Phi\|_{L^\infty}
\|n-1\|_{L^2(\Omega)} \|u\|_{L^2(\Omega)}.
\end{split}
\end{equation*}
Due to $u|_{\partial\Omega}=0$,  the Poincar\'{e} inequality implies the existence of some $C_1>0$ such that
\begin{equation*}
\int_\Omega|u|^2\leq C_1\int_\Omega|\nabla u|^2.
\end{equation*}
By the Young inequality, we have
\begin{equation}\label{4.36}
\begin{split}
\frac d{dt}\int_\Omega|u|^2+C_2\int_\Omega|\nabla u|^2+ C_2\int_\Omega| u|^2 \leq
C_3\|n-1\|^2_{L^2(\Omega)}
\end{split}
\end{equation}
for some $C_2>0$ and $C_3>0$.

Integrating  \eqref{4.36} over $(t_1,t_2)$ and using \eqref{4.33}, we arrive at
\begin{equation}\label{4.37}
\begin{split}
&\frac{1}{t_2-t_1}\int^{t_2}_{t_1}\|u(\cdot,s)\|^2_{W^{1,2}(\Omega)}ds\\
\leq & \frac{C_4}{t_2-t_1} \int_\Omega |u(\cdot,t_1)|^2
+
\frac{C_4}{t_2-t_1} \int^{t_2}_{t_1} \int_\Omega (n-1)^2
\\
\leq &
\frac{C_4}{t_2-t_1} \int_\Omega |u(\cdot,t_1)|^2
+
\frac{C_5}{t_2-t_1} \int_\Omega(n(\cdot,t_1)-1-\ln n(\cdot,t_1))
+
C_5 \sup_{s\geq t_1} \int_\Omega \frac {|\nabla v(\cdot,s)|^2}{v(\cdot,s)}
\end{split}
\end{equation}
for some $C_4>0$ and $C_5>0$.
Finally, thanks to \eqref{4.31}, one can arrive at \eqref{4.35} by the argument in Lemma \ref{lemma4.9}.
\end{proof}

\vspace{.5cm}

\it {Proof of Theorem 1.1 (ii).}\rm   \quad  The existence of global classical solutions and their boundedness have been shown in Lemma \ref{lemma4.5}.  The large-time behavior of the solutions  as claimed in Theorem 1.1 (ii) has been proven in Lemmas \ref{lemma4.8}--\ref{lemma4.10}.

\vspace{.7cm}

{\bf Conflict of interest}:
No potential conflict of interest is reported by the authors.

{\bf Ethics approval}:
Ethics approval is not required for this research.

{\bf Funding}: The authors are grateful to the referee for detailed and
insightful comments, which led to significant improvements to this paper.
This work was partially supported by the National Natural Science Foundation of China (No.~12071030, 12271186  third author, No.~12371444 first author) and NUS Academic Research Fund (no.~A-0004282-00-00; second and third authors). The second author also acknowledges the hospitality of the Fields Institute for Research in Mathematical Sciences in Toronto where part of this research was carried out.

{\bf Data availability statement}:
All data that support the findings of this study are included within the article (and any supplementary
files).

\end{document}